\documentclass[11pt]{article}
 \usepackage[utf8]{inputenc}
 \usepackage[T1]{fontenc}  
\usepackage[english]{babel}
\usepackage{amsfonts}

\usepackage{amsmath,amsxtra,amsthm,latexsym,mathabx}
\usepackage{amssymb} % \Subset !
\usepackage{accents}

\usepackage{hyperref}

\newtheorem{thm}{Theorem}[section]
 \newtheorem{cor}[thm]{Corollary}
 \newtheorem{lem}[thm]{Lemma}
 \newtheorem{prop}[thm]{Proposition}

 \theoremstyle{definition}
 \newtheorem{defn}{Definition}[section]
 \theoremstyle{remark}
 \newtheorem{rem}{Remark}[section]

 \numberwithin{equation}{section}

\renewcommand{\>}{\rangle}

\newcommand{\al}{\alpha}

\newcommand{\la}{\lambda}

\newcommand{\ga}{\gamma}
\newcommand{\Ga}{\Gamma}

\newcommand{\Om}{\Omega}

\newcommand{\si}{\sigma}

 \newcommand{\CC}{\mathbb{C}} 
  
 \newcommand{\GG}{\mathbb{G}} 
\newcommand{\HH}{\mathbb{H}} \newcommand{\II}{\mathbb{I}} 
 
\newcommand{\LL}{\mathbb{L}}  
\newcommand{\NN}{\mathbb{N}}

\newcommand{\RR}{\mathbb{R}}

\newcommand{\XX}{\mathbb{X}}

\usepackage[bbgreekl]{mathbbol}
\DeclareSymbolFontAlphabet{\mathbbm}{bbold}
\DeclareSymbolFontAlphabet{\mathbb}{AMSb}%
\DeclareMathAlphabet{\mathmybb}{U}{bbold}{m}{n}
\newcommand{\ab}{\alpha}

\newcommand{\one}{\mathbbm{1}}

\newcommand{\Bc}{\mathcal{B}}
\newcommand{\Cc}{\mathcal{C}}

\newcommand{\Ec}{\mathcal{E}}

\newcommand{\Ic}{\mathcal{I}}

\newcommand{\Lc}{\mathcal{L}}

\newcommand{\Tc}{\mathcal{T}}
\newcommand{\Uc}{\mathcal{U}}

\newcommand{\Zc}{\mathcal{Z}}

 \newcommand{\Hf}{\mathfrak{H}}

 \newcommand{\Sf}{\mathfrak{S}} \newcommand{\Tf}{\mathfrak{T}} 
\newcommand{\Xf}{\mathfrak{X}} \newcommand{\Yf}{\mathfrak{Y}}
\newcommand{\af}{\mathfrak{a}} \newcommand{\bfr}{\mathfrak{b}}

\newcommand{\mf}{\mathfrak{m}}
\newcommand{\nf}{\mathfrak{n}}

\newcommand{\Nr}{\mathrm{N}}

\newcommand{\rr}{\mathrm{r}}

\newcommand{\ubf}{\mathbf{u}}
\newcommand{\vbf}{\mathbf{v}}

\newcommand{\wt}{\widetilde}
\newcommand{\wh}{\widehat}

\newcommand{\<}{\langle}
\newcommand{\uph}{\upharpoonright}
\newcommand{\imb}{\hookrightarrow}

\newcommand{\ii}{\mathrm{i}}
\newcommand{\ee}{\mathrm{e}}
\newcommand{\dd}{\mathrm{d}}
\newcommand{\pa}{\partial}

\newcommand{\disc}{\mathrm{disc}}
\newcommand{\ess}{\mathrm{ess}}
\newcommand{\sym}{\mathrm{sym}}

\DeclareMathOperator{\im}{Im}
\DeclareMathOperator{\re}{Re}
\DeclareMathOperator{\dom}{dom}

\DeclareMathOperator{\Gr}{Gr}
\DeclareMathOperator{\Div}{div}
\DeclareMathOperator{\gradm}{\mathbf{grad}}

\DeclareMathOperator{\gradn}{\mathbf{grad}_0}
\DeclareMathOperator{\Divn}{\Div_0}

\DeclareMathOperator{\Hom}{{\mathcal{LH}}}

\newcommand{\Ao}{A}

\newcommand{\A}{\mathcal{A}}

\newcommand{\gan}{\ga_{\mathrm{n}}}

\newcommand{\x}{\mathbf{x}}
\newcommand{\n}{\mathbf{n}}
\newcommand{\Hs}{\Hf}

\newcommand{\D}{\Om}

\newcommand{\cross}{{\natural}}
\newcommand{\mul}{\mathcal{M}}
\newcommand{\mulz}{{\mul_{\scriptstyle \zeta}}}

\begin{document}
\title{M-dissipative generalized impedance  boundary conditions, discrete spectra, and pointwise multipliers between fractional Sobolev spaces} 
\author{}
\date{}%\today}
\maketitle
  
{ \center{\large Illya M. Karabash%$^\text{ a}$
\\[2ex] 
}   }  
  
 {\small 
\noindent Institute for Applied Mathematics, University of Bonn, Endenicher Allee 60,
53115 Bonn, Germany;\\[1ex]
% ORCID: https://orcid.org/0000-0003-2296-5711\\[2ex]
E-mails: ikarabas@uni-bonn.de, i.m.karabash@gmail.com\\}

\medskip

\vspace{4ex}

\begin{abstract}
The paper studies properties of acoustic operators in bounded Lipschitz domains $\Omega$ with m-dissipative  generalized impedance boundary conditions.
We prove that such acoustic operators have compact resolvent if and only if the impedance operator from the trace space $H^{1/2} (\partial \Omega)$ to the other trace space $H^{-1/2} (\partial \Omega)$ is compact. This result is applied to the question of the discreteness of the spectrum and to the particular cases of damping and impedance boundary conditions. The method of the paper is based on abstract results written in terms of boundary tuples and is applicable to other types of wave equations.
\end{abstract}

\vspace{1ex}
{\small
\noindent
MSC2020-classes: 
%60H25 %  	Random operators and equations (aspects of stochastic analysis)
%35Q61 %   	Maxwell equations
35F45   %	Boundary value problems for systems of linear first-order PDEs
%47B80   %	Random linear operators
35P05   %	General topics in linear spectral theory for PDEs
%35J56   %	Boundary value problems for first-order elliptic systems
%78A25  %   	Electromagnetic theory, general
58J90 %   	Applications of PDEs on manifolds
47B44 % Accretive operators, dissipative operators, etc.
47F10  % 	Elliptic operators and their generalizations 
%34G10 %   	Linear differential equations in abstract spaces
47D03 %   	Groups and semigroups of linear operators
\\[0.5ex]
}

\medskip

{\noindent\small \textsc{Acknowledgement.} 
The author's research was supported by the Heisenberg Programme (project 509859995) of the German Science Foundation (DFG, Deutsche Forschungsgemeinschaft), by the Hausdorff Center for Mathematics funded by the German Science Foundation under Germany's Excellence Strategy -- EXC-2047/1 -- 390685813, 
and by the CRC 1060 `The mathematics of emergent effects' funded by the German Science Foundation at the University of
Bonn. 
}

\section{\label{s:i}Introduction}

The main goal of this paper is to study the discreteness of spectra of  eigenvalue problems 
\begin{gather}\label{e:EqEl}
- \nabla \cdot ( \al^{-1} (\x) \nabla p (\x))  =  \la^2 \beta (\x) p (\x)  ,  \qquad \x \in \D \subset \RR^d, \quad d \ge 2,  
\end{gather}
equipped on the boundary $\pa \D$ of a bounded domain $\D$ with dissipative boundary conditions of the form
\begin{gather} \label{e:IBCEl}
\ii \la \Zc \ga_0 (p) =  \gan ( \ab^{-1} \nabla p)  ,
 % \Zc \ga_0 (p) = \frac{1}{\ii \la} \gan ( \ab^{-1} \nabla p)  ,
\end{gather}
where $\ga_0 : p \mapsto p\!\uph_{ \pa \D}$ and $\gan : \vbf  \mapsto \n \cdot (\vbf \uph_{\pa \D})$ are the scalar and normal traces, respectively. The boundary $\pa \D$ is supposed to have the Lipschitz regularity.
%, and so,   the outward unit vector-field $\n (\x)$ normal to $\partial \D$ belongs to the space $L^\infty  (\pa \D, \RR^d)$.
We assume that the coefficients  $\ab (\cdot)  = (\ab_{j,k} (\cdot))_{j,k=1}^d \in L^\infty (\D, \RR_{\sym}^{d\times d})$ and $\beta (\cdot) \in L^\infty (\D,\RR)$ are uniformly positive in $\D$ (see Sections \ref{s:SymAcOp}-\ref{s:OpGenIBC} for the detailed formulation of the problem). 
%In particular, the differential operation $- \beta^{-1} \Div \al^{-1} \gradm $ is uniformly elliptic.

The \emph{impedance operator} $\Zc : \dom \Zc \subseteq H^{1/2} (\pa \D) \to H^{-1/2} (\pa \D)$ in \eqref{e:IBCEl} is  supposed to be \emph{accretive}
as an operator from  $H^{1/2} (\pa \D)$ to $H^{-1/2} (\pa \D)$
in the sense that $\re \<\Zc h,h\>_{L^2 (\pa \D)} \ge 0$ for every $h$ in the domain $\dom \Zc$ (of definition) of $\Zc$,
where $\< \cdot,\cdot\>_{L^2 (\pa \D)} $ in the standard sesquilinear pairing of $H^{-1/2} (\pa \D)$ and $H^{1/2} (\pa \D)$ obtained by the continuation from the inner  product $(\cdot|\cdot)_{L^2 (\pa \D)}$ of the complex Hilbert space $L^2 (\pa \D)=L^2 (\pa \D,\CC)$. By $\re z$
and $\im z$ we denote the real and imaginary parts of $z \in \CC$.
%Note that, in these settings, the operators 
%$\ga_0  : H^{1} (\D) \to H^{1/2} (\pa \D)$ and $\gan : \HH (\Div, \D) \to H^{-1/2} (\pa \D)$ are continuous.
%\vbf=\frac{1}{\ii \la} \ab^{-1} \gradm p

Equation \eqref{e:EqEl} stems from the time-harmonic formulation for acoustic \cite{L13,KZ15,S21} and dimensionally reduced Maxwell equations \cite{FK96,ACL18}. The boundary condition \eqref{e:IBCEl} is one of the forms of generalized impedance boundary conditions (GIBCs), see \cite{HJN05,K24} (for GIBCs in the context of Maxwell systems, see also \cite{ACL18,EK22}).

The well-known \emph{impedance boundary condition} \cite{L83,HJN05,CK13}
\begin{gather} \label{e:Imp}
\ii \la \zeta (\x) p (\x) =  \n (\x) \cdot (\al^{-1} \nabla p (\x)) , \qquad \x \in \pa \D
\end{gather}
 is the  particular case of  \eqref{e:EqEl} that appears if the impedance operator $\Zc$ is the operator of multiplication $\mul_{\zeta}$ on a function $\zeta:\pa \D \to \overline{\CC}_\rr$ measurable with respect to (w.r.t.) the surface measure of $\pa \D$. Here $\overline{\CC}_\rr := \{z \in \CC : \re z  \ge 0\}$
is the closure of the complex open right half-plane $\CC_\rr := \{z \in \CC : \re z  > 0\}$. 
The function $\zeta$ is called \emph{impedance coefficient}.

In the Mechanics context, \eqref{e:Imp} bears the name of the boundary damping condition. In this case, the boundary damping coefficient $\zeta (\cdot)$ is usually assumed to have values in $\overline{\RR}_+ := [0,+\infty)$ \cite{CZ95}. 
The impedance boundary condition kindred to \eqref{e:Imp}  is used with Maxwell systems, in particular, to model leaky optical cavities, and is called often the Leontovich boundary condition \cite{LL84,K94,LL04,YI18,EK22,EK24}. In the context of Electrodynamics, the complex values of the impedance coefficient $\zeta (\cdot)$ appears naturally, e.g., if the cavity region $\D$ is surrounded by a weakly conductive medium. Values of $\zeta $ with close to $(-\ii )\RR_+ = \{ -\ii c  : c \in (0,+\infty) \}$ can serve as an approximation for a superconducting outer medium \cite{LL84}.

The question of the discreteness of the spectrum for the problem \eqref{e:EqEl}-\eqref{e:IBCEl} requires an appropriate interpretation. Indeed, \eqref{e:EqEl}-\eqref{e:IBCEl} allows one to define associated eigenvalues. However, the discreteness of the spectrum is defined essentially as the emptiness of the essential spectrum. Since the spectral parameter $\la$ enters into the boundary condition \eqref{e:IBCEl}, it is difficult to define the essential spectrum for \eqref{e:EqEl}-\eqref{e:IBCEl} without an appropriate operator reformulation. 

One of operator reformulations for the problem \eqref{e:EqEl}-\eqref{e:IBCEl} is connected with generators of acoustic semigroups and can be written in the block-matrix form as
\begin{gather} \label{e:BMEq}
 \ii \begin{pmatrix}  0 & I  \\
- \beta^{-1} \Div \al^{-1} \gradm & 0  \end{pmatrix} \begin{pmatrix} u  \\ p   \end{pmatrix} = \la \begin{pmatrix} u  \\ p   \end{pmatrix} , 
\end{gather}
with the boundary condition 
\begin{gather} \label{e:GIBC}
\Zc  \ga_0 (p) =  \gan (-\al^{-1} \nabla u)  , \quad \x \in \pa \D .
\end{gather}

In the phase space $\XX$ consisting of pairs $\Phi = \{u,p\}$ and equipped with the `energy norm'
\[
\|\Phi \| = \left( \int_\D \al^{-1} |\nabla u|^2  + \int_\D \beta | p|^2  \right)^{1/2} ,
\]
one can associate with \eqref{e:BMEq}-\eqref{e:GIBC} an operator $\Bc_\Zc$ (see the details in Section \ref{s:SymAcOp}-\ref{s:OpGenIBC}).
The operator $\Bc_\Zc = \Bc_{\Zc,\al,\beta}$ depends on $\D$, $Z$, $\al$, $\beta$. However,
the Lipschitz  domain $\D$ and the coefficients $\al$ and $\beta$ are assumed to be fixed throughout the paper.
Therefore we keep only the impedance operator $\Zc$ in the notation for $\Bc_\Zc$.

Due to the assumption of accretivity of $\Zc$,
the operator $\Bc_\Zc$ is dissipative in the Hilbert space $\XX$ in the sense that $\im (\Bc_\Zc \Phi | \Phi)_\XX \le 0$ for all $\Phi \in \dom \Bc_\Zc$. The conventions \cite{E12,EK22} that we use for dissipative, accretive, and m-dissipative operators are described in Section \ref{s:AbsMdis}. 

The special role of m-dissipative operators $\Bc_\Zc$ in modeling of lossy (dissipative)  resonators is that they provide the characterization of the case where $(-\ii) \Bc_\Zc$ is 
a generator of a contraction semigroup  $\{\ee^{-\ii t \Bc_\Zc}\}_{t>0}$. The particular case of a conservative resonator corresponds to a selfadjoint operator $\Bc_\Zc$, and so, to the unitary group $\{\ee^{-\ii t \Bc_\Zc}\}_{t\in \RR}$. 

It is often expected that a model of a dissipative resonator in a bounded domain $\D$ should lead to an operator with a purely discrete spectrum.
The main result of this paper is the following characterization.

\begin{thm} \label{t:main}
The following two statements are equivalent:
\begin{itemize}
\item[(i)]  the acoustic operator $\Bc_\Zc$ is an m-dissipative operator in $\XX$ with a compact resolvent;
\item[(ii)] the impedance operator $\Zc$ is accretive and compact as an operator from $H^{1/2} (\pa \D)$ to $H^{-1/2} (\pa \D)$.
\end{itemize}
If statements (i)-(ii) hold true, then  $\Bc_\Zc$ has a purely discrete spectrum.
\end{thm}

This theorem is proved in Section \ref{s:ProofMT}. 

If the adjoint $\Zc^\cross:  H^{1/2} (\pa \D) \to H^{-1/2} (\pa \D)$ to $\Zc$ w.r.t. the pairing $\< \cdot,\cdot\>_{L^2 (\pa \D)} $ satisfies $\Zc^\cross = - \Zc$,
the operator $\Bc_\Zc$ in Theorem \ref{t:main} becomes selfadjoint in $\XX$ and the compactness of $\Zc : H^{1/2} (\pa \D) \to H^{-1/2} (\pa \D)$ is equivalent to the discreteness of the spectrum of $\Bc_\Zc$, see Corollary \ref{c:SiDisc}.

The proof of Theorem \ref{t:main} is based on abstract operator theoretic results written in terms of m-boundary tuples. An m-boundary tuple \cite{EK22} is a generalization of the notion of boundary triple \cite{K75,GG91,DM95,P12,DM17,BHdS20} adapted to the specifics of partial differential operators (PDOs), see  Section \ref{s:BT}. Boundary tuples of various types allows one to write in abstract form linear boundary value problems for PDEs, see \cite{GG91,BL07,M10,AGW14,BM14,KZ15,DM17,S21,BGM22,DHM22,EK22,K24} and Section \ref{s:AbsMdis} below.
%An abstract versions of Theorem \ref{t:main} are given in Sections \ref{s:CompResAbs}-\ref{s:PabsCR}.

If $\Zc = \mulz$ is an operator of  multiplication on a measurable coefficient $\zeta:\pa \D \to \overline{\CC}_\rr$, then GIBC  \eqref{e:GIBC} becomes the impedance boundary condition of the form 
\begin{gather} \label{e:IBCint}
\zeta (\x)  p (\x) =  - \n \cdot (\al^{-1} \nabla u (\x))  , \qquad \x \in \pa \D.
\end{gather}
In this case, the criterion of Theorem \ref{t:main} says that \emph{the acoustic operator 
$\Bc_{\mulz}$ corresponding to \eqref{e:IBCint} is m-dissipative with compact resolvent if and only if $\mul_\zeta$ is a compact pointwise multiplier from $H^{1/2} (\pa \D)$ to $H^{-1/2} (\pa \D)$}.

For impedance boundary conditions, we obtain the following sufficient condition. 

\begin{thm} \label{t:ImpLq}
Consider the acoustic operator 
$\Bc_{\mulz}$ corresponding to \eqref{e:BMEq}, \eqref{e:IBCint} and 
assume that $\zeta \in L^q (\pa \D, \overline{\CC}_\rr)$ for a certain $q>d-1$. 
Then $\Bc_{\mulz}$ is an m-dissipative operator in $\XX$ with a compact resolvent and a purely discrete spectrum.
\end{thm}

This theorem is a part of a more general result proved in Section \ref{s:SingImp}, see Theorem \ref{t:SingImp}. In particular, Section \ref{s:SingImp} provides rigorous definitions of the multiplication operator $\mulz$ in the case of an unbounded impedance coefficient $\zeta (\cdot)$.
In combination with the rigorous interpretation of GIBCs given in Section \ref{s:OpGenIBC} and Definition \ref{d:Imp}, this gives a rigorous realization for an associated m-dissipative  acoustic operator $\Bc_{\mulz}$.
Note that, for Maxwell systems,  there exist various nonequivalent interpretations of imped\-ance boundary conditions, see \cite{K94,LL04,EK22}.

General dissipative acoustic boundary conditions were studied in \cite{KZ15,S21}. A parametrization and rigorous randomizations of m-dissipative acoustic boundary conditions were  obtained recently in \cite{K24}.
It should be stressed that there is a great variety of generalizations of impedance boundary conditions for various wave equations.
We have considered in this paper a particular version of GIBCs that is independent of time and independent of the spectral parameter (from the point of view of equation \eqref{e:BMEq}).
For other types of GIBCs, absorbing boundary conditions, and non-reflecting boundary conditions, we refer to \cite{KG89,YI18,FJL21,MAGB22} and references therein.  

\textbf{Notation.} 
Let $\Xf$, $\Xf_1$, and $\Xf_2$ be (complex) Hilbert spaces and $\Yf$, $\Yf_1$, and $\Yf_2$ be Banach spaces. 
By $\Lc (\Yf_1,\Yf_2)$ we denote the Banach space of bounded (linear) operators $T : \Yf_1 \to \Yf_2$, while 
$\Hom (\Xf_1,\Xf_2)$ is the set of \emph{linear homeomorphisms} from $\Xf_1$ to $\Xf_2$.
Besides, $\Lc (\Yf) := \Lc (\Yf,\Yf)$ and $\Hom (\Xf) := \Hom (\Xf,\Xf)$.
 By $I_\Yf$ (by $0_{\Lc(\Yf)}$) the identity operator (resp., the zero operator) in $\Yf$ is denoted, although  the subscript is dropped  if the space $\Yf$ is clear from the context. In the notation for the resolvent 
$(T-\la)^{-1}=(T-\la I)^{-1}$,   the identity operator $I$ is often skipped.

By $\one$ we denote the constant function equal to $1$. For $1 \le p \le \infty$, $\Sf_p = \Sf_p (\Xf)$  are the Schatten-von-Neumann ideals of compact operators in $\Xf$.
By $\Gr T := \{\{f,Tf\} \ : \ f \in \dom T\}$ we denote the graph of the operator $T$, and consider it as a normed space with the graph norm \cite{Kato,GG91}. We  use the natural identification of $\Gr T$ and the domain (of definition) $\dom T$ of $T$. 
%So $\dom T$ is also considered as a normed space with the graph norm.
The spectrum 
%and the set of eigenvalues 
of an operator $T:\dom T \subseteq \Xf \to \Xf$ is denoted by $\si (T)$.
% and $\si_\pr (T)$, respectively. 
The notation $\rho (T) = \CC \setminus \si (T)$ stands for the resolvent set, whereas 
$\si_\ess (T)$ and $\si_\disc (T)$ denote the essential and the discrete spectra of $T$, respectively (see Section \ref{s:CompResAbs} for the definitions).

%A symmetric operator $T$ is called nonnegative (and we write $T \ge 0$) if $(Tf|f)_{\Xf} \ge 0$ for all $f \in \dom T$. The notation $T_1 \le T_2$ means $T_2 -T_1 \ge 0$.  A symmetric operator is uniformly positive if $T \ge \de I$ for a certain constant $\de>0$. 
% For $T \in \Lc (\Xf)$, $\re T = \frac12(T+T^*)$ and 
%$\im T = \frac{1}{2\ii} (T - T^*)$ are the real and imaginary parts of $T$, respectively.

\section{Abstract impedance operators and resolvent compactness}
\label{s:AbsIO}

We consider the problem of the compactness of the resolvent in the abstract settings that 
fit to many types of partial differential operators (PDOs). In what follows, $\Xf$, $\Xf_1$, and $\Xf_2$ are (complex) Hilbert spaces. 
The notation $T:\dom T \subseteq \Xf_1 \to \Xf_2$ means that a (linear) operator $T$ is considered as an operator from $\Xf_1$ to $\Xf_2$
defined on a domain $\dom T  \subseteq \Xf_1$  that is a linear (possibly non-closed) subspace of $\Xf_1$.
If $\dom T = \Xf_1$, we simplify this notation to $T: \Xf_1 \to \Xf_2$.

Let  $A$ be a closed symmetric densely defined  operator  in $\Xf$.
In order to define abstract boundary conditions complementing its adjoint operator $A^*$, one needs  abstract analogues of the spaces of boundary values and the abstract integration by parts. These are introduced in the next subsection.
 
\subsection{Boundary tuples and generalized rigged Hilbert spaces} 
\label{s:BT}

The triple $(\Hs_{-,+},\Hs,\Hs_{+,-})$ consisting of (complex) Hilbert spaces is called a \emph{generalized rigged Hilbert space} 
if spaces $\Hs_{\mp,\pm}$ are dual to each other with respect to the $\Hs$-pairing $\<\cdot,\cdot\>_\Hs$ that is obtained from the inner product $(\cdot|\cdot)_\Hs$ of the \emph{pivot space} $\Hs$ as an extension by continuity to 
$\Hs_{\mp,\pm} \times \Hs_{\pm,\mp}$, see \cite[Appendix to IX.4, Ex.~3]{RSII} and \cite{S21,EK22}. This notion is convenient for the description of duality of trace spaces arising for various wave equations \cite{S21,EK22}. Note that a generalized rigged Hilbert space does not assume any of embeddings between spaces $\Hs_{\mp,\pm}$ and $\Hs$. 

\begin{rem}  \label{r:rHs}
A generalized rigged Hilbert space $(\Hs_+,\Hs,\Hs_-)$ with the additional assumption that continuous embeddings $\Hs_+ \imb \Hs \imb \Hs_-$ hold is a rigged Hilbert space
(or a Gelfand triple) in the standard sense. 
\end{rem}

\begin{defn}[m-boundary tuple, \cite{EK22}] \label{d:MBT}
Assume that auxiliary Hilbert spaces $\Hs$ and $\Hs_{\mp,\pm}$ form  a generalized rigged Hilbert space $(\Hs_{-,+},\Hs,\Hs_{+,-})$. 
The tuple $(\Hs_{-,+},\Hs, \Hs_{+,-}, \Ga_0,\Ga_1)$ is called an \emph{m-boundary tuple}  for the operator $\Ao^*$
if the following conditions hold:
\item[(i)] the map 
$\Ga : f \mapsto \{ \Ga_0 f , \Ga_1 f \} $ is a surjective linear operator from $\dom \Ao^*$ onto $\Hs_{-,+} \oplus \Hs_{+,-}$; 
\item[(ii)] 
\begin{equation} \label{e:AIP}
\text{$(\Ao^*f |g)_{\Xf} - (f |\Ao^*g)_{\Xf} =  \< \Ga_1 f , \Ga_0 g \>_\Hs  - \<\Ga_0 f , \Ga_1 g \>_\Hs $ for all $f,g \in \dom \Ao^*$.}
\end{equation}
\end{defn}

Note that  $\< \Ga_1 f , \Ga_0 g \>_\Hs$  in \eqref{e:AIP} is understood in the sense of the sesquilinear pairing 
$\< h_{+,-}  , h_{-,+}  \>_{\Hs}$ of $\Hs_{+,-}$ and $\Hs_{-,+}$ (with $h_{\pm,\mp} \in \Hs_{\pm,\mp}$) that is generated by the sesquiline\-ar inner
product $(\cdot|\cdot)_\Hs$ of the pivot space $\Hs$. We use the same notation 
$\< \cdot  , \cdot \>_{\Hs}$ for the sesquilinear pairing of   $\Hs_{-,+}$ with $\Hs_{+,-}$, which is the case of the term $\<\Ga_0 f , \Ga_1 g \>_\Hs$
in  \eqref{e:AIP}.
By $\Hs_{-,+} \oplus \Hs_{+,-}$, we denote the orthogonal sum of the spaces $\Hs_{\pm,\mp}$.

If additionally the continuous embeddings $\Hs_{-,+} \imb \Hs \imb \Hs_{+,-}$ of a rigged Hilbert space take place in Definition \ref{d:MBT},    
$(\Hs_{-,+},\Hs,\Hs_{+,-},\Ga_0,\Ga_1)$ is called a \emph{rigged boundary tuple} for  $\Ao^*$ \cite{K24} 
(cf. \cite[Section 6]{DM95} and \cite{BM14}, where somewhat similar constructions appeared in a less explicit form).
In the particular case of the trivial duality $\Hs_{-,+} = \Hs = \Hs_{+,-}$, an m-boundary tuple $(\Hs_{-,+},\Hs,\Hs_{+,-},\Ga_0,\Ga_1)$ becomes effectively a boundary triple $(\Hs,\Ga_0,\Ga_1)$ in the standard sense of \cite{K75,GG91} (see also \cite{DM95,P12,DM17,BHdS20}). 

\begin{defn}[boundary triple, \cite{K75,GG91}] \label{d:BT}
An auxiliary Hilbert space $\Hs$ and 
the maps $\wh \Ga_j: \dom \Ao^* \to \Hs$, $j=0,1$, form a \emph{boundary triple} $(\Hs, \wh \Ga_0, \wh \Ga_1)$ for $\Ao^*$ if  $\wh \Ga : f \mapsto \{ \wh \Ga_0 f ,\wh \Ga_1 f \} $  is a surjective linear operator from $\dom \Ao^*$ onto $\Hs \oplus \Hs$ and
\begin{gather*} \label{e:AIPBT}
\text{$(\Ao^*f|g)_\Xf - (f|\Ao^*g)_\Xf =  (\wh \Ga_1 f | \wh \Ga_0 g )_\Hs  - (\wh \Ga_0 f | \wh \Ga_1 g )_\Hs $ for all $f,g \in \dom \Ao^*$.   }
\end{gather*}
\end{defn}

The operators $\Ga_0$ and $\Ga_1$ in Definition \ref{d:MBT} are abstract analogues of traces (boundary maps) of the theory of PDOs, while \eqref{e:AIP} is the abstract integration by parts.
Similarly to the theory of boundary triples, it is easy to see \cite{EK22} that 
\begin{equation} \label{e:domA}
\dom \Ao = \{\psi \in \dom \Ao^* \ : \ 0 = \Ga_0 \psi = \Ga_1 \psi\} .
\end{equation}
%and  that the operators $\Ga_0 : \dom \Ao^* \to \Hs_{-,+}$ and $\Ga_1 : \dom \Ao^* \to \Hs_{+,-}$ are bounded, where 
%$\dom \Ao^*$ is perceived as a Hilbert space with the graph norm $\|\psi \|_{\dom \Ao^*} = (\|\psi\|_\Xf^2+\|\Ao^* \psi \|_\Xf^2)^{1/2}$.

Boundary triples were introduced  for the description of general boundary conditions in the cases of ordinary differential equations (ODEs) or system of ODEs (see \cite{K75,GG91,DM17} for the history of the problem and additional references). They do not fit exactly to the specifics of the integration by parts for PDOs since the same auxiliary space $\Hs$ is used as the target space for the abstract trace maps $\Ga_0$ and $\Ga_1$. In order to address this difficulty, a number of regularizations and generalizations of boundary triples were suggested \cite{GG91,DM95,BL07,P12,AGW14,BM14,KZ15,S21,BGM22,DHM22,EK22}. In particular, the abstract definition of m-boundary tuple was  introduced in \cite{EK22} and was applied there to the trace spaces of  Maxwell systems.

\subsection{Abstract impedance boundary conditions and m-dissipativity}
\label{s:AbsMdis}

In the rest of this section, we assume that   
\begin{multline}
 \text{$\Tf=(\Hs_{-,+},\Hs, \Hs_{+,-}, \Ga_0,\Ga_1)$ is an m-boundary tuple for the adjoint $\Ao^*$}\\
 \text{ to a closed densely defined symmetric operator $A$.}
\label{e:aBT}
\end{multline}
The existence of an m-boundary tuple is equivalent to the statement that 
the deficiency indices $\nf_\pm (A)$ of $A$ are equal to each other.
In the case of assumption \eqref{e:aBT}, $\dim \Hs = \dim \Hs_{\mp,\pm} = \nf_\pm (A)$, where 
$\dim \Hs$ denotes the dimensionality of a Hilbert space $\Hs$ \cite{GG91,EK22}.

A operator $Z: \dom Z \subseteq \Hs_{-,+} \to \Hs_{+,-} $ acting from the space $\Hs_{-,+}$ to the space $\Hs_{+,-}$ and having a domain $\dom Z \subseteq \Hs_{-,+} $ is called \emph{accretive} if $\re \<Z y , y \>_{\Hs} \ge 0$ for all $y \in \dom Z$ \cite{EK22}.
An accretive operator $Z: \dom Z \subseteq \Hs_{-,+} \to \Hs_{+,-} $ is called \emph{maximal accretive} if it has no proper accretive extensions 
$\wh Z:\dom \wh Z \subseteq \Hs_{-,+} \to \Hs_{+,-} $. Recall that an extension $ \wh Z$ of $Z$ is \emph{proper} if $\dom \wh Z \supsetneqq \dom Z$.
In the case where $\Hs_{-,+} = \Hs_{-,+} = \Hs$ and $ \< \cdot  , \cdot \>_{\Hs} = (\cdot|\cdot)_{\Hs}$, these definitions  become the standard definitions for accretive and maximally accretive operators in a Hilbert space $\Hs$, see \cite{P59,Kato}.

The following definition of an abstract impedance boundary condition can be seen as a generalization of \cite[Section 7.1]{EK22}.

\begin{defn} \label{d:Imp}
Assume that $Z:\dom Z \subseteq \Hs_{-,+} \to \Hs_{+,-}$ is an accretive operator. 
Let the operator $A_Z = A^*\uph_{\dom A_Z}$ be the restriction  of $A^*$ to 
the set $\dom A_Z$ consisting of all $y \in \dom A^*$ such that $\Ga_0 y \in \dom Z$ and 
\begin{gather} \label{e:IBC}
Z \Ga_0 y = \ii \Ga_1 y  .
\end{gather}
Then we say that the restriction $A_Z$ of $A^*$ is generated by an \emph{abstract impedance boundary condition}  \eqref{e:IBC}. The operator $Z$ is called an \emph{(abstract) impedance operator}
associated with $A_Z$. 
\end{defn}

This correspondence between $Z$ and $A_Z$  assumes that $A^*$ and the m-boundary tuple $\Tf$ for $A^*$ are already fixed. Note that \eqref{e:domA} implies that 
$A_Z$ is always an extension of $A$.

In the context of GIBCs for acoustic and Maxwell systems, impedance operators were discussed and studied in \cite{HJN05,ACL18,EK22,EK24,K24}.

A (linear) operator $T:\dom T \subseteq \Xf \to \Xf$ in a Hilbert space $\Xf$ is called 
\emph{dissipative} if $\im (Ty|y)_{\Xf} \le 0$ for all $y \in \dom T$.
An  operator $T$ is called  \emph{m-dissipative} if $\CC_+ := \{z \in \CC : \im z >0\}$ 
is a subset of its resolvent set $\rho (T)$
and 
$\| (T- z)^{-1}\| \le (\im z)^{-1}$ for all $z \in \CC_+ $, see, e.g., \cite{E12,EK22}.
The set of m-dissipative operators is a subset in the class of dissipative operators \cite{Kato}. 
An operator $T$ is called m-accretive if $(-\ii) T$ is m-dissipative.

In these settings a spectrum $\si (T)$ of an m-dissipative operator lies in the closure 
$\overline{\CC}_- $ of the lower complex half-plane $\CC_- := \{\la \in \CC : \im \la <0\}$,
and a spectrum of an m-accretive operator lies in $\overline{\CC}_\rr := \{\la \in \CC : \re \la \ge 0\}$.
Let us notice that there exists a variety of other  conventions \cite{P59,GG91,Kato,KZ15,S21}.

\begin{rem} \label{r:CritMdis}
It follows from the results of Phillips \cite{P59} that the following statements are equivalent for an operator $T:\dom T \subseteq \Xf \to \Xf$:
\begin{itemize}
\item[(i)] $T$ is m-accretive, %$(-\ii)T$ is a generator of a contraction semigroup,
\item[(ii)] $T$ is a closed maximal accretive operator, 
\item[(iii)] $T$ is a densely defined maximal accretive  operator.
\end{itemize}
Obviously, the analogous equivalences connect the notions of m-dissipative and maximal dissipative operators.
\end{rem}

Recall that an  operator $K$ belonging to the Banach space $\Lc (\Xf)$ of bounded operators in $\Xf$
is called a \emph{contraction} if $\| K \| \le 1$.

\begin{rem} \label{r:Gen}
The importance of the class of m-dissipative operators in the context of wave equations is that 
they describe dissipative resonators. This corresponds to the following  criterion of Phillips:
an operator $T$ is m-dissipative if and only if $(-\ii) T$ is a generator of a contraction semigroup  \cite{P59,Kato,E12}. 
\end{rem}

Assume that an operator $T:\dom T \subseteq \Hs_{-,+} \to \Hs_{+,-}$ has a domain $\dom T$ dense in $\Hs_{-,+}$. Then the operator $T^\cross:\dom T^\cross \subseteq \Hs_{-,+} \to \Hs_{+,-}$ is defined as the adjoint to $T$ 
w.r.t. the $\Hs$-pairing $\<\cdot,\cdot\>_\Hs$ (see Section \ref{s:BT}). We  say for brevity that  $T^\cross$ is the \emph{$\cross$-adjoint operator to $T$}.
This definition of the $\cross$-adjoint means that its graph $\Gr T^\cross = \{\{f,T^\cross f\} : f \in \dom T^\cross\}$ 
consist of all $\{g_{-,+},g_{+,-}\} \in \Hs_{-,+} \oplus \Hs_{+,-}$ such that $\<h_{+,-} | g _{-,+}\>_\Hs  =  \<h_{-,+} | g_{+,-}\>_\Hs$ for all  $\{ h_{-,+} , h_{+,-} \} \in \Gr T$.
Consequently, the operator $T^\cross$ is closed.
If $T=T^\cross$, we say that $T$ is $\cross$-selfadjoint operator.

In the case of the trivial duality $\Hs_{\mp,\pm} = \Hs$, the  definition of $T^\cross$ becomes the standard definitions for the adjoint operator $T^*$ in $\Hs$. 
Note that, in the case $\Hs_{\mp,\pm} \neq \Hs$ and $T:\dom T \subseteq \Hs_{-,+} \to \Hs_{+,-}$, the adjoint operator $T^*: \dom T^* \subseteq \Hs_{+,-} \to \Hs_{-,+}$ is  formally different from the $\cross$-adjoint  $T^\cross:\dom T^\cross \subseteq \Hs_{-,+} \to \Hs_{+,-}$, but there is a 1-to-1 correspondence between these types of adjoint operators that follows from the standard identifications of bounded linear functionals on $\Hs_{\mp,\pm}$.

If an operator $T:\dom T \subseteq \Hs_{-,+} \to \Hs_{+,-}$ is densely defined in $\Hs_{-,+}$ and
is  closable, then its closure satisfies $\overline{V} = (V^\cross)^\cross$.

The assumption of accretivity  imposed on all impedance operators  $Z$ in Definition \ref{d:Imp} 
guaranties that the associated restriction $A_Z$ (of $A^*$) is a dissipative operator in $\Xf$.
This follows directly from  the abstract integration by parts \eqref{e:AIP}.
The accretivity of $Z$  is not sufficient to guarantee the m-dissipativity of $A_Z$.  
The corresponding criterion of the m-dissipativity  of $A_Z$ is provided by Proposition \ref{p:m-dis} below,
which is a generalization of \cite[Corollary 2.3]{EK22},  \cite[Corollary 7.2]{EK22}, and \cite[Theorem 2.3]{K24}.

\begin{prop} \label{p:m-dis} 
In the settings of Definition \ref{d:Imp}, the following sta\-tements are equivalent:
\begin{itemize}
\item[(i)] The operator $A_Z$ is m-dissipative in $\Xf$.
\item[(ii)] The impedance operator $Z$ is closed and maximal accretive.
\item[(iii)] The impedance operator $Z$ is densely defined (in $\Hs_{-,+}$) and maximal accretive.
\end{itemize}
Moreover, $A_Z = A_Z^*$ if and only if $Z^\cross = - Z$.
\end{prop}

\begin{proof} The proof can be obtained from Remark \ref{r:CritMdis} and the theory of boundary tuples in the way completely analogous to the proof of
\cite[Corollary 7.2]{EK22}, which gives a similar result for the particular case of abstract Maxwell operators.
\end{proof}

\subsection{Compactness for  resolvents and for impedance operators}
\label{s:CompResAbs}

An eigenvalue $\la$ of $T: \dom T \subseteq \Xf \to \Xf$ is called \emph{isolated} if $\la$ is an isolated point of the spectrum $\si (T)$ of $T$. The \emph{discrete spectrum} $\si_\disc (T)$ of $T$ is the set of isolated eigenvalues of $T$ with finite algebraic multiplicities \cite{Kato,RSIV}. We say that $T$ has  \emph{purely discrete spectrum} if $\si(T) = \si_\disc (T)$. The closed set $\si_\ess (T) = \si (T) \setminus \si_\disc (T)$ is called an \emph{essential spectrum} of $T$ \cite{RSIV}.

Assume that $\la_0$ belongs to the resolvent set $\rho (T)$ of $T$ and the resolvent $(T-\la_0 )^{-1} = (T-\la_0 I)^{-1}$ at $\la_0$ is a compact operator (for brevity, we sometimes omit the identity operator $I$ from resolvent-type notations). Then the resolvent $(T-\la)^{-1}$ is compact for every $\la \in \rho (T)$; in this case, it is said that $T$ is an \emph{operator with compact resolvent}. An operator $T$ with compact resolvent has purely discrete spectrum, see, e.g., \cite{Kato}. 

Definition \ref{d:MBT} implies that, for $j=0,1$, the restrictions 
\begin{gather} \label{e:A01}
\wh A_j = A^* \uph_{\ker \Ga_j}  \text{are selfadjoint extensions of $A$.} 
\end{gather}

The following theorem is our main technical tool for the subsequent sections.

\begin{thm}\label{t:absCR}
Assume that $\wh A_1$ has a compact resolvent.
Let an m-dissipative restriction $A_Z$ of $A^*$ be generated by an abstract impedance boundary condition $ Z \Ga_0 y = \ii \Ga_1 y $ (see Proposition \ref{p:m-dis}).
Then the following statements are equivalent: 
\begin{itemize}
\item[(i)]  $A_Z$ has a compact resolvent;
\item[(ii)] the impedance operator $Z$ is compact (as an operator from $\Hs_{-,+}$ to $ \Hs_{+,-}$).
\end{itemize} 
\end{thm}

The proof is given in Subsection \ref{s:PabsCR}. It is based on the reduction to the case of 
a boundary triple, where a variety of compatibility results for the compactness properties are known, see \cite[Theorems 3.1 and 3.4]{GG91} and 
\cite[Theorem 7.106 and Corollary 7.107]{DM17}.

If $T$ is a selfadjoint or normal operator, then
$\si(T) = \si_\disc (T)$ is equivalent to the compactness of the resolvent of $T$.

We see now that Theorem \ref{t:absCR} and Proposition \ref{p:m-dis} imply immediately the following result on the discreteness of the spectrum of $A_Z$.

\begin{cor} \label{c:absSiDisc}
Assume that $\wh A_1$ has a compact resolvent.
Let $A_Z$ be the restriction of $A^*$ associated with an impedance operator $Z$ as before. Then the following statements hold true: 
\begin{itemize}
\item[(i)]  If $Z$ is compact, then $\si(A_Z) = \si_\disc (A_Z)$. 
\item[(ii)] If $Z^\cross = - Z$, then the compactness of $Z$ is equivalent to $\si(A_Z) = \si_\disc (A_Z)$.
\end{itemize} 
\end{cor}

\subsection{Compatibility results and proof of Theorem \ref{t:absCR}}
\label{s:PabsCR}

We assume \eqref{e:aBT} in this subsection.
Every m-boundary tuple can be reduced to a boundary triple in a number of nonequivalent ways. Roughly speaking, each choice of orthogonal coordinates in $\Hs$ produces its own reduction.
This can be seen from Proposition \ref{p:BT} below, which a modification of \cite[Proposition 6.1]{EK22}. Proposition \ref{p:BT} requires an extension of a notion of $\cross$-adjoint operator.

The definition of  $\Hs$-pairing-adjoint operators $T^\cross$ of Section \ref{s:AbsMdis} can be naturally extended \cite{EK22} to operators acting between spaces $\Hs$ and $\Hs_{\mp,\pm}$. For example, for an operator $V$ belonging to the Banach space $\Lc (\Hs, \Hs_{-,+})$ of bounded operators from $\Hs$ to $\Hs_{-,+}$, there exists a unique $\cross$-adjoint operator $V^\cross \in 
\Lc (\Hs_{+,-}, \Hs)$ such that 
\begin{equation} \label{e:HsAd}
\text{$\< V f,g\>_\Hs = (  f | V^\cross g)_\Hs$   for all $ f \in \Hs$ and $g \in \Hs_{+,-}$.}
\end{equation}

Let us denote by $\Hom (\Xf_1,\Xf_2)$ a set of linear homeomorphisms from 
$\Xf_1$ to $\Xf_2$. Then, for $V \in \Hom (\Hs,\Hs_{-,+})$, one has $V^\cross \in \Hom (\Hs_{+,-}, \Hs)$ and $(V^\cross)^{-1} = (V^{-1})^\cross$.

In what follows, we fix a certain $V \in \Hom (\Hs, \Hs_{-,+})$.

\begin{prop}[\cite{EK22}]\label{p:BT}
 Let $(\Hs_{-,+},\Hs,\Hs_{+,-}\Ga_0,\Ga_1)$ be an m-boundary tuple for $A^*$. Then
 \[
\text{$ \Tf^V = (\Hs, \Ga_0^V, \Ga_1^V) := (\Hs, V^{-1} \Ga_0, V^\cross \Ga_1) $ 
\quad and \quad $\Tf^V_* = (\Hs, \ii \Ga_1^V,  (-\ii) \Ga_0^V)$}
\]
 are  boundary triples for $A^*$. We say that the boundary triples $\Tf^V$ and $\Tf^V_*$
 are dual to each other.
\end{prop}

The following result, which describes all m-dissipative extensions of the symmetric operator $A$, is essentially a part of \cite[Corollary 6.6]{EK22} up to minor reformulations (it can be obtained as a combination of Proposition \ref{p:BT} with the Kochubei description  \cite{K75} of maximal dissipative linear relations in $\Hs$).

\begin{prop}[\cite{EK22}] \label{p:absM-dis} 
The following statements are equivalent:
\begin{itemize}
\item[(i)] $\wh A$ is an m-dissipative extension of $A$;
\item[(ii)] there exists a contraction $K$ in $\Hs$ such that $\wh A$ is the restriction of $A^*$ that is generated by  the `abstract boundary condition' 
 \begin{gather} \label{e:K+IGa}
(K+I_\Hs) V^{-1} \Ga_0 y + \ii (K-I_\Hs) V^\cross  \Ga_1 y = 0 
\end{gather}
in the sense that $\dom \wh A = \{y \in \dom A^* : \text{ \eqref{e:K+IGa} is satisfied }\}$;
\item[(iii)] $\wh A$ is an m-dissipative restriction of $A^*$.
\end{itemize}
This equivalence establishes a 1-to-1 correspondence between m-dissipative restrictions $\wh A$ of $A^*$ and contractions $K$ in $\Hs$. Besides, $\wh A = \wh A^*$ if and only if $K$ is unitary.
\end{prop}

For $1 \le  p \le \infty$, we denote by $\Sf_p (\Xf)$  the Schatten-von-Neumann ideals of compact operators in a Hilbert space $\Xf$. In particular, $\Sf_\infty (\Xf)$ is the set of all compact operators in $\Xf$.

\begin{prop}[\cite{K24}] \label{p:R-RinS}
In the settings of Proposition \ref{p:absM-dis}, let us consider two contractions $K_j \in \Lc(\Hs)$, $j=1,2$,
and two corresponding m-dissipative restrictions $\wt A_j$, $j=1,2$, of $A^*$ defined by the abstract boundary conditions \eqref{e:K+IGa} with $K=K_j$, $j=1,2$.
Let $\la \in \rho (\wt A_1 ) \cap \rho (\wt A_2 ) $ and $1 \le p \le \infty$. 
Then
\[
 \text{$(\wt A_2 - \la)^{-1} - (\wt A_1 - \la)^{-1}  \in \Sf_p (\Xf)$ if and only if $K_2 - K_1 \in \Sf_p (\Hs)$.}
 \]
\end{prop}

This result follows from Proposition \ref{p:BT} and  the corresponding result for boundary triples \cite[Theorem 3.1]{GG91}. Note that the operators  $\wt A_1$ and  $\wt A_2$ are m-dissipative in Proposition \ref{p:R-RinS}, and so, $\rho (\wt A_1 ) \cap \rho (\wt A_2 ) \supseteq \CC_+$.

\emph{Now, our aim is to prove Theorem \ref{t:absCR} using Proposition \ref{p:R-RinS}}.
Let $A_Z$ be an m-dissipative restriction of $A^*$ generated by the abstract impedance boundary condition
$ Z \Ga_0 y = \ii \Ga_1 y $. This means that the impedance operator $Z:\dom Z \subseteq \Hs_{-,+} \to \Hs_{+,-}$ satisfies statements  (ii)-(iii) of Proposition \ref{p:m-dis}. Assume also that the selfadjoint restriction $\wh A_1$ of $A^*$ generated by the condition $\Ga_1 y = 0$ has compact resolvent. Note that 
\begin{equation} \label{e:K1=I}
\text{$\wh A_1$ can be defined as in Proposition \ref{p:absM-dis} with $K$ replaced by $K_1 = -I$}.
\end{equation} 

 In order to use Proposition \ref{p:R-RinS}, we need to transform the condition $ Z \Ga_0 y = \ii \Ga_1 y $
into the form \eqref{e:K+IGa} with a certain contraction $K \in \Lc (\Hs)$. Such a contraction $K$ exists and is unique  by Proposition \ref{p:absM-dis}. Let us find this $K$. 

Using the homeomorphisms $V \in \Hom (\Hs, \Hs_{-,+})$ and $V^\cross \in \Hom (\Hs_{+,-},\Hs)$, one can rewrite $ Z \Ga_0 y = \ii \Ga_1 y $ as 
$Z_V V^{-1}\Ga_0 y = \ii V^\cross  \Ga_1 y $ with $Z_V = V^\cross Z V $. Since, by Proposition \ref{p:m-dis}, $Z$ is a maximal accretive and closed operator from $\Hs_{-,+}$ to $\Hs_{+,-}$, we see that $Z_V$ is a maximal accretive and closed operator in $\Hs$. By Remark \ref{r:CritMdis}, $Z_V$ is an m-accretive operator in $\Hs$. 
Then the Cayley transform 
\[
\Cc_{Z_V} := (Z_V- I)( Z_V+I)^{-1}
\] 
of $Z_V$ is a contraction in $\Hs$ \cite{P59}.

%Let us show that $K = \Cc_{Z_V}$. 
Since
\begin{gather} \label{e:CZV-I}
 \text{ the operator $\Cc_{Z_V}-I = -2( Z_V+I)^{-1}$ is injective }
\end{gather} 
and $(\Cc_{Z_V}-I)^{-1} = - ( Z_V+I)/2$, 
the condition \eqref{e:K+IGa} with $K=\Cc_{Z_V} $ can be equivalently transformed into  the condition
 \begin{gather*} %\label{e:K+IGa}
(I-\Cc_{Z_V})^{-1} (I+\Cc_{Z_V}) V^{-1} \Ga_0 y =  \ii V^\cross  \Ga_1 y ,
\end{gather*}
and further into $Z_V V^{-1}\Ga_0 y = \ii V^\cross  \Ga_1 y $. Indeed, $(I-\Cc_{Z_V})^{-1} (I+\Cc_{Z_V}) $ is the inverse Cayley transform of $\Cc_{Z_V}$, and so, 
$Z_V = (I-\Cc_{Z_V})^{-1} (I+\Cc_{Z_V}) $.

Summarizing, we see that the restriction $A_Z$  is generated by the condition \eqref{e:K+IGa} with $K=\Cc_{Z_V}$.

Under the assumptions of Theorem \ref{t:absCR}, the selfadjoint restriction $\wh A_1$ of $A^*$ generated by the condition $\Ga_1 y = 0$ has a compact resolvent. %Note that $\wh A_1$ can be defined as in Proposition \ref{p:absM-dis} with $K$ replaced by $K_1 = -I$. 
Using \eqref{e:K1=I} and applying Proposition \ref{p:R-RinS} to the contractions $K_1 = -I$ and $K_2 = \Cc_{Z_V}$, we see that 
\begin{gather} \label{e:CZV+I}
\text{$\wh A$ has a compact resolvent if and only if \quad  $\Cc_{Z_V} +I \in \Sf_\infty (\Hs)$.}
\end{gather}

\begin{lem} \label{l:SinfEq}
(i) $\Cc_{Z_V} +I \in \Sf_\infty (\Hs)$ if and only if $Z_V \in \Sf_\infty (\Hs)$.
\item[(ii)] $Z_V \in \Sf_\infty (\Hs)$ if and only if $Z$ is compact as an operator from $\Hs_{-,+}$
to $\Hs_{+,-}$.
\end{lem}
\begin{proof}
(i) \emph{Step 1.} Note that $\Cc_{Z_V} +I = 2 Z_V (Z_V+I)^{-1} $, where $(Z_V+I)^{-1} \in \Lc (\Hs)$ 
since the spectrum of the m-accretive operator $Z_V$ is in $\overline{\CC}_\rr$.
Hence, $Z_V \in \Sf_\infty (\Hs)$ implies $\Cc_{Z_V} +I \in \Sf_\infty (\Hs)$.

\emph{Step 2.} Assume now that $T=\Cc_{Z_V} +I$ is compact. The formula 
$Z_V = (I-\Cc_{Z_V})^{-1} (I+\Cc_{Z_V}) $ gives $Z_V = (2I - T)^{-1} T$, 
where the operator $2I - T = I-\Cc_{Z_V}$ is injective due to \eqref{e:CZV-I}.
This and the assumption $T \in \Sf_\infty (\Hs)$ imply that $2 \not \in \si (T)$, and so,
$(2I - T)^{-1} \in \Lc (\Hs)$. Thus, $Z_V = (2I - T)^{-1} T \in \Sf_\infty (\Hs)$.

(ii) We use the fact that $V$ and $V^\cross$ are homeomorphisms, $V \in \Hom (\Hs, \Hs_{-,+})$ and $V^\cross \in \Hom (\Hs_{+,-},\Hs)$. Thus, the equality $Z_V = V^\cross Z V $ implies statement (ii).
\end{proof}

Combining \eqref{e:CZV+I} and Lemma \ref{l:SinfEq}, \emph{we complete the proof of Theorem \ref{t:absCR}.}

\begin{rem}
Alternatively, Theorem \ref{t:absCR} can be obtained from Proposition  \ref{p:BT} and \cite[Corollary 7.107]{DM17}.
\end{rem}

\section{M-dissipative acoustic operators with discrete spectra}

\subsection{Boundary tuples for second- and first- order acoustic operators}
\label{s:SymAcOp}

Let $\D$ be a domain in $\RR^d$ with $d \ge 2$, i.e., $\D$ is a non-empty bounded open connected subset of $\RR^d$.  We always  assume that $\D$ is a Lipschitz domain, i.e.,
its  boundary $\pa \D$ satisfies the Lipschitz regularity condition \cite{M03,ACL18}.
Let $\RR_{\sym}^{d\times d}$ be a normed space of $d\times d$ real-valued symmetric matrices (with an arbitrary norm).

Let the material parameters $\ab (\cdot)  = (\ab_{j,k} (\cdot))_{j,k=1}^d \in L^\infty (\D, \RR_{\sym}^{d\times d})$ and $\beta \in L^\infty (\D,\RR)$ be uniformly positive in the sense that they satisfy for almost all (a.a.) $\x \in \D$ the conditions 
\[
\text{
 $\beta (\x) \ge \beta_0 $ and $\ab (\x) \ge c \II $ for certain constants $\beta_0, c >0$.}
\] 
  Here
$\ab (\x)$ and the identity $d\times d$-matrix $\II$ are identified with selfadjoint operators
in $\CC^d$, while $\ge$ is the standard partial order of bounded selfadjoint operators. 
By $\ab^{-1} (\cdot) \in L^\infty (\D, \RR_{\sym}^{d\times d}) $, we denote the pointwise inversion  $ (\ab (\x))^{-1}$, $\x \in \D$, for the matrix-valued function $\ab (\cdot)$.
The uniform positivity assumption for $\ab (\cdot)$ implies that $\ab^{-1}  \in L^\infty (\D, \RR_{\sym}^{d\times d})$.

By $\LL^2 (\D)=L^2 (\D,\CC^d)$, we denote the standard Hilbert space of complex vector fields in $\D$ equipped with the sesquilinear inner product 
\[
(  \ubf | \vbf )_{\LL^2 } =  \int_{\D} \ubf \cdot \overline{\vbf}  =  \int_{\D} (  \ubf  | \vbf  )_{\CC^d} . 
\]
The weighted Hilbert space $L^2_\beta = L^2_\beta (\D)$ of $\CC$-valued functions 
coincide with $L^2 (\D)=L^2 (\D,\CC) $ as a linear space, but has another (equivalent) norm $  \| \cdot \|_{L^2_\beta}$ satisfying 
\[
\| f  \|_{L^2_\beta}^2 =  ( \beta f | f)_{L^2 (\D) }^{1/2} = \int_{\D} \beta |f|^2  .
\]

For $k \in \NN$, we use  the standard complex Hilbertian Sobolev spaces 
$H^k ( \D) = W^{k,2} ( \D)$ and $H^k_0 ( \D) = W^{k,2}_0 ( \D)$ (see, e.g., \cite{M03,GM11}),
as well as the corresponding Hilbert spaces of distributions $H^{-k} ( \D) = W^{-k,2} ( \D)$ that are 
dual to  $H^k_0 ( \D) $ w.r.t. the pivot space $ H^0 (\D) := L^2 (\D)$.

Let $s \in \RR$, $1<p<\infty$, and $1\le q \le \infty$.
Following  \cite{T83,T92}, we denote by $B^s_{p,q} ( \RR^d) $  the Besov spaces in $\RR^d$.
Note that  the standard  fractional Sobolev(-Slobodetckij) spaces
$W^{s,2} ( \RR^d) $ coincide with $B^s_{2,2} (\RR^d)$ up to equivalence of norms \cite{T83}.

Let $s>0$, $1<p<\infty$, and $1\le q \le \infty$.
 Let $B^s_{p,q} (\overline{\D}) $ and $B^s_{p,q} (\pa \D) $ be the Besov spaces in the closure $\overline{\D}$ of $\D$ and on $\pa \D$, respectively, in the sense of \cite{JW84}. Note that $\overline{\D}$  is a $d$-set in the terminology of \cite{JW84}, while $\pa \D$ is a $(d-1)$-set. By comparison of definitions and results concerning Besov spaces $B^s_{p,q} (\overline{\D}) $ and  $B^s_{p,q} (\D) $ in \cite{JW84,T92,T02}, one sees that $B^s_{p,p} (\D) = B^s_{p,p} (\overline{ \D})$. For $0<s<1$, the fractional Sobolev spaces $W^{s,2} ( \D) $  coincide with   $B^s_{2,2} (\D) = B^s_{2,2} (\overline{ \D})$  up  to equivalence of norms (this can be seen by the comparison of \cite{M03} and \cite[Section 5.1.1]{JW84}). 

On the Lipschitz boundary $\pa \D$, we use for $1 \le q \le \infty$ the spaces 
 $L^q (\pa \D) $  of
$\CC$-valued $L^q$-functions built w.r.t.  the surface measure of $\partial \D$ and equipped with 
the standard $L^q$-norms $ \| \cdot \|_{L^q (\pa \D)}$.  
The outward unit vector $\n (\x)$ normal to $\partial \D$ at $\x$ is defined for a.a. $\x \in \pa \D$. The resulting function $ \n (\cdot)$ belongs to the space $L^\infty  (\pa \D, \RR^n)$ of the $\RR^n$-vector fields.

 For $s \in (0,1]$ and $1<p<\infty$,  one can define \cite{M03,GM11} the complex fractional Sobolev spaces $W^{s,p} (\pa \D)=W^{s,p} (\pa \D,\CC)$ lifting $W^{s,p} (\RR^{d-1})$ to $\pa \D$ via localization  and pullback. These spaces are reflexive Banach spaces. It follows from \cite{GMMM11} that in the particular case $H^s (\pa \D) := W^{s,2} (\pa \D)$, $s \in (0,1]$, an equivalent norm can be choosen in such a way that  $H^s (\pa \D)$ becomes a  Hilbert space. 
 By $H^{-s} (\pa \D) $ we denote the adjoint space to $H^s (\pa \D)$  w.r.t. the pivot space $H^0 (\pa \D) := L^2 (\pa \D)$.
We use the compact embeddings 
\begin{gather} \label{e:CEs1s2}
H^{s_2} (\pa \D) \imb \imb H^{s_1} (\pa \D) , \qquad -1 \le s_1 < s_2 \le 1  \qquad \text{(see \cite{GMMM11})}.
\end{gather}

% In the Hilbert space $(\LL^2_{\ab,\beta} (\D), \| \cdot \|_{\LL^2_{\ab,\beta}})$ we consider for $t >0$ the evolution equation of the linear acoustic corresponding to :
%\begin{gather} \label{e:Ac}
%\pa_t \Psi (\x,t)  = \begin{pmatrix} 0 & - \ab^{-1} \gradm \\
%- \beta^{-1} \Div & 0 \end{pmatrix} \Psi (\x,t) \quad \text{ with } \Psi (\cdot,t) = \begin{pmatrix} \vbf (\cdot,t) \\ p (\cdot,t) \end{pmatrix} \in \LL^2_{\ab,\beta} (\D),
%\end{gather} 
%where $\vbf (\cdot,t) \in \LL^2_\ab (\D) $ and $ p (\cdot,t) \in L^2_\beta (\D)$ for each $t \ge 0$.

The gradient operator $\gradm : f \mapsto \nabla f $ with the maximal (natural) domain 
$H^1 (\D) := \{ f \in L^2 (\D) \ : \nabla f \in \LL^2 (\D) \}$ in $L^2 (\D)$
is considered  an operator from $L^2 (\D)$ to $\LL^2 (\D)$, i.e., $\gradm : H^1 (\D) \subset L^2 (\D)  \to \LL^2 (\D)$.
Here $\nabla f$ is understood in the distribution sense. Note that the operator $\gradm$ is closed.

Vector spaces built as domains of operators are assumed to be equipped with the graph norms. 
If these operators are closed as operators between Hilbert spaces, their domains also become  Hilbert spaces.

The operator $\gradn$ is defined as the closure 
$\overline{\gradm \uph}_{C^\infty_0 (\D)}$ 
in $L^2 (\D)$ of the restriction $\gradm \uph_{C^\infty_0 (\D)}$ of the operator $\gradm$ to the space $C^\infty_0 (\D)$ of compactly supported in $\D$ smooth complex scalar functions. Its domain 
$\dom \gradn = H^1_0 (\D) $  is a closed subspace of $H^1 (\D)$.

The divergence operator 
$
\Div : \ubf \mapsto \nabla \cdot u$ 
%\qquad \Div: \dom (\Div) \subset \LL^2 (\D) \to L^2 (\D)
from $\LL^2 (\D)$ to $L^2 (\D)$  is considered on its maximal domain  
\[
 \HH (\Div, \D) := \{ \ubf \in \LL^2 (\D)  : \nabla \cdot \ubf \in L^2 (\D)\}.
 \]
This operator is closed and coincides with  $(-\gradn)^*$.

The operator of the scalar trace 
$\ga_0 : p \mapsto p\!\uph_{ \pa \D}$ can be understood as a continuous operator  $\ga_0 \in \Lc (H^1(\D), H^{1/2} (\pa \D))$ and is surjective in this sense, i.e., $\ga_0 H^1(\D) = H^{1/2} (\pa \D)$.
Similarly, the normal trace   
$\gan : \vbf  \mapsto \n \cdot \vbf (x) \uph_{\pa \D}$ can be understood as 
 $\gan \in \Lc (\HH (\Div, \D), H^{-1/2} (\pa \D))$ and  is also surjective , i.e., $\gan \HH (\Div, \D) = H^{-1/2} (\pa \D)$ 
 (see, e.g., \cite{M03}).

The closed operator $\Div_0 : \HH_0 (\Div,\D) \subset \LL^2 (\D)   \to L^2 (\D)$ 
defined by $\Div_0 := \gradm^*$
is a restriction of the operator $\Div$ to 
\[
\text{ $\HH_0 (\Div,\D) = \{\ubf \in \HH (\Div, \D) \ : \ \gan \ubf = 0\}$ (see, e.g., \cite{L13}).} 
 \]
 
Our main aim is to study the discreteness of spectra of operators associated with models of leaky acoustic resonators. Following \cite{L13}, we write the 2nd-order acoustic evolution equation in the Schrödinger form
\begin{gather} \label{e:Ac2or}
\ii \pa_t \wt \Phi = \bfr_{\ab,\beta} (\wt \Phi) \quad \text{ with  } \quad \bfr_{\ab,\beta} : \Phi = \begin{pmatrix} u  \\ p   \end{pmatrix} \mapsto \ii \begin{pmatrix}  p  \\
\beta^{-1}  \nabla \cdot (  \ab^{-1}  \nabla u)  \end{pmatrix} 
%\begin{pmatrix} p  \\
%\beta^{-1} \nabla \cdot (\ab^{-1} \nabla u)  \end{pmatrix} 
%\\ 
%\begin{pmatrix} 0 & I  \\
%\beta^{-1}  \Div  \ab^{-1}  \gradm & 0 \end{pmatrix} \Phi , \quad \text{where $\Phi  = \begin{pmatrix} u  \\ p   \end{pmatrix}$ },
\end{gather} 
The energy of the state $\wt \Phi (t) = \{\wt u (t), \wt p (t)\} $ at the time $t$ is given by 
\[
%\frac12 \| \Psi \|^2 =
\Ec (\wt \Phi (t)) = \frac 12 \left( \int_\D \al^{-1} |\nabla \wt u|^2 \dd x + \int_\D \beta |\wt p|^2 \dd x \right) .
\]

In the next subsection we equip $\ii \pa_t \wt \Phi = \bfr_{\ab,\beta} (\wt \Phi)$ with a linear time-independent boundary condition in such a way that $\Ec (\wt \Phi (t)) \le \Ec (\wt \Phi (0))$ for $t>0$.
In other words, we are interested in boundary conditions for the equation 
$ \pa_t \wt \Phi = - \ii \bfr_{\ab,\beta} (\wt \Phi)$ that lead to a contraction semigroup  in the appropriate phase space with the `energy norm'. This is equivalent to the m-dissipativity of the acoustic operator $\wh  \Bc$ associated with such  a boundary condition.

Let us define the Hilbert space corresponding to the energy norm. Consider in the Sobolev space $H^1 (\D) $ the semi-norm $\| u \|_{1,\ab^{-1}} = (\ab^{-1} \nabla u|\nabla u)_{\LL^2 }^{1/2}$. Since the subspace 
\[
\gradm H^1 (\D) :=\{\nabla u : u \in H^1 (\D)\}
\]
 is closed in $\LL^2 (\D)$ (see, e.g., \cite[Section 7.4]{L13}), one sees that  the factorization of the semi-Hilbert space $(\gradm H^1 (\D), \| \cdot \|_{1,\al^{-1}} )$ w.r.t. the 1-dimensional subspace $\{c \one : c \in \CC\}$ of constant functions produces a  Hilbert space, which we denote by 
 $(H_{1,\ab^{-1}} (\D), \| \cdot \|_{1,\al^{-1}}) $.
 
The evolution equation $\ii \pa_t \wt \Phi  = \bfr_{\ab,\beta} (\wt \Phi)$ 
is considered in the phase space 
\[
\XX := H_{1,\ab^{-1}} (\D) \oplus L^2_{\beta} (\D),
\]
which is a Hilbert space constructed as an orthogonal sum of $H_{1,\ab^{-1}} (\D)$ 
and $ L^2_{\beta} (\D)$. This means that
\[
\| \Phi \|_\XX^2 = \| u \|_{1,\ab^{-1}}^2 + \| p  \|_{L^2_\beta}^2  \quad \text{ for $\Phi = \{u,p\}$  corresponds to $2 \Ec (\Phi)$}.
\]

In order to introduce general m-dissipative boundary conditions by means of the extension theory and m-boundary tuples (see Section \ref{s:AbsMdis}),
one has to associate with the differential operation $\bfr_{\ab,\beta}$ a symmetric operator $\Bc:\dom \Bc \subset \XX \to \XX$  with a minimal natural domain $\dom \Bc$ in $\XX$.
We introduce this operator  $\Bc$ by the operator block-matrix 
\begin{gather}
\Bc := \ii \begin{pmatrix} 0 & \Ic_0  \\
\beta^{-1} \Div_0  \ab^{-1} \gradm_1  & 0  \end{pmatrix} , \label{e:B}
\end{gather}
where $\gradm_1 \in  \Lc (H_{1,\ab^{-1}} (\D), \LL^2 (\D))$ is one more version of the gradient operator  that maps an equivalence class $ \{u +c \one : c \in \CC\} \in H_{1,\ab^{-1}} (\D)$
to $ \nabla u$, while  $\Ic_0: H^1_0 (\D) \subset L^2_{\beta} (\D) \to H_{1,\ab^{-1}} (\D) $ is   the identification operators that map $p \in H^1_0 (\D)$ to the equivalence class $ \{p +c \one : c \in \CC\}$ in $H_{1,\ab^{-1}} (\D)$.
The notation $\one$  stands to the constant function equal to $1$.

This definition means that 
\begin{gather} \label{e:domB}
\dom \Bc =\{ \{u,p\} \in \XX \  : \quad \ab^{-1} \nabla u \in \HH_0 (\Div,\D) , \quad p \in   H^1_0 (\D)\}.
\end{gather}

It is easy to see by the integration by parts  that the operator $\Bc$ is symmetric (w.r.t. the inner product $(\cdot|\cdot)_\XX$ of $\XX$). Further properties of $\Bc$ are summarized in the following lemma, which is a combination of \cite[Lemma 2.19]{K24}, \cite[Remark 2.9]{K24}, \cite[Proposition 3.3]{K24}, and \cite[Proposition 6.2]{K24}.

\begin{lem}[\cite{K24}] \label{l:B}
(i) The operator $\Bc$ is symmetric, closed, and densely defined in $\XX$.
\item[(ii)]  The adjoint operator $\Bc^*$ is given by the formulae
\begin{gather}
 \Bc^* = \ii 
\begin{pmatrix} 0 & \Ic_1  \\
\beta^{-1} \Div  \ab^{-1} \gradm_1  & 0  \end{pmatrix},  \label{e:B*}
\\
 \dom \Bc^* =\{ \{u,p\} \in \XX : \ab^{-1} \nabla u \in \HH (\Div,\D) , \ p \in H^1 (\D)  \},
\label{e:domB*}
\end{gather}
where 
$\Ic_1:H^1 (\D) \subset L^2_{\beta} (\D) \to H_{1,\ab^{-1}} (\D)$ is the identification operators that map $p \in H^1 (\D)$ to the equivalence class $ \{p +c \one : c \in \CC\}$ in $H_{1,\ab^{-1}} (\D)$.
\item[(iii)] 
Let us define  $\wt \ga_0  (\{u,p\}) := \ga_0 (p) $ and $\wt \gan (\{u,p\}) := \gan (-\ab^{-1} \nabla u)$. Then
\begin{multline} \label{e:tTuples}
 \Tc^2 :=(H^{1/2} (\pa \D), L^2 (\pa \D) , H^{-1/2} (\pa \D),  \wt \ga_0, \ (-\ii) \wt \gan ) 
\\ \text{ and } \quad \Tc_*^2 :=(H^{-1/2} (\pa \D), L^2 (\pa \D) , H^{1/2} (\pa \D),  \wt \gan, (- \ii) \wt \ga_0  ) 
\end{multline} 
are m-boundary tuples for $\Bc^*$ (we say that these two m-boundary tuples are mutually dual).
\end{lem}

Note that m-boundary tuples are defined in Definition \ref{d:MBT} and that the boundary space $L^2 (\pa \D) $  plays in \eqref{e:tTuples} the role of a pivot space. Its scalar product $(\cdot|\cdot)_{L^2 (\pa \D)} $ generates for each $s \in (0,1]$ two sesquilinear forms that provide the parings of the Sobolev space $H^{-s} (\pa \D)$ with $H^{s} (\pa \D) $,  and conversely, of $H^{s} (\pa \D)$ with $H^{-s} (\pa \D) $. These two sesquilinear forms are denoted in the same way $\<\cdot,\star\>_{L^2 (\pa \D)}$ and are called the $L^2 (\pa \D)$-pairings (for all $s \in (0,1]$).  We mainly use these $L^2 (\pa \D)$-pairings for  the rigged Hilbert space $H^{1/2} (\pa \D) \imb  L^2 (\pa \D) \imb H^{-1/2} (\pa \D)$ and the m-boundary tuple $ \Tc^2$ of \eqref{e:tTuples}.

\subsection{First-order acoustic operators and the proof of Lemma \ref{l:B}}
\label{s:1stOrder}

In this subsection, a sketch of the proof of Lemma \ref{l:B} is given following \cite{K24}.
It explains the connection with the 1st-order acoustic operators of \cite{L13}, which will be used in Section \ref{s:ProofMT}.

The `weighted' Hilbert space of vector fields $ \LL^2_\ab = \LL^2_\ab (\D)$ coincides with the Hilbert space 
$\LL^2 (\D) =L^2 (\D,\CC^d) $ as a linear space, but is equipped with the equivalent weighted norm $ \| \cdot \|_{\LL^2_\ab}$ defined by
$  \| \vbf  \|_{\LL^2_\ab}^2 =  ( \ab \vbf | \vbf)_{\LL^2 (\D) }$.
The Hilbert space $\LL^2_{\ab,\beta}$ is defined as the orthogonal sum
\[
\LL^2_{\ab,\beta} (\D)  = \LL^2_{\ab,\beta} := \LL^2_{\ab} (\D) \oplus L^2_\beta (\D).
 \]

Let us consider in the phase space $\LL^2_{\ab,\beta} (\D)$  the 1-st-order version of the acoustic system 
\begin{gather} \label{e:Ac}
\ii \pa_t \wt \Psi   = \af_{\ab,\beta} \wt \Psi , 
\quad 
 \af_{\ab,\beta} :  \begin{pmatrix} \vbf  \\ p \end{pmatrix} \mapsto \frac{1}{\ii} \begin{pmatrix}  \ab^{-1} \nabla p \\
 \beta^{-1} \nabla \cdot \vbf \end{pmatrix}.
\end{gather} 
In what follows, we use the notation $\Psi = \{\vbf,p\} = \begin{pmatrix} \vbf  \\ p \end{pmatrix} $.

We associate with the 1st-order differential expression $\af_{\ab,\beta}$ the closed symmetric operator $\A: \dom \A \subset \LL^2_{\ab, \beta} (\D) \to \LL^2_{\ab, \beta} (\D)$ defined by 
\begin{gather}\label{e:A}
\A \begin{pmatrix} \vbf \\ p \end{pmatrix} = \begin{pmatrix} 0 & -\ii \ab^{-1} \gradn \\
-\ii \beta^{-1} \Divn & 0 \end{pmatrix} \begin{pmatrix} \vbf \\ p \end{pmatrix} , \quad 
\end{gather}
where the state $\Psi = \{  \vbf , p \}$ in the phase space $\LL^2_{\ab, \beta} (\D)$ belongs to the domain of $\A$ given by $\dom \A := \HH_0 (\Div,\D) \times  H_0^1 (\D) $.
Obviously, $\dom \A$ is dense in $\LL^2_{\ab, \beta} (\D) $. 

It is easy to see from the integration by parts formula 
\begin{gather} \label{e:IbyP}
( \gradm p | \vbf)_{L^2 (\D,\CC^3)} + (p | \Div \vbf)_{L^2 (\D)} = \<\ga_0 p | \gan \vbf \>_{L^2 (\pa \D)}  
\end{gather}
with $ p \in H^1 (\D)$ and $\vbf \in \HH (\Div, \D)$ (see \cite{M03,ACL18}) 
that the adjoint operator $\A^*$ has another domain $\dom \A^* = \HH (\Div,\D) \times H^1 (\D)$, but is associated with the same differential expression $\af_{\ab,\beta}$.

Let us define  $\wh \ga_0  (\{\vbf,p\}) := \ga_0 (p) $ and $\wh \gan (\{\vbf,p\}) := \gan (\vbf)$. 
The integration by parts \eqref{e:IbyP} implies also that 
\begin{multline} \label{e:mTf}
 \Tc^1 :=(H^{1/2} (\pa \D), L^2 (\pa \D) , H^{-1/2} (\pa \D),  \wh \ga_0, \ (-\ii) \wh \gan ) \text{ and }
\\  \quad \Tc_*^1 :=(H^{-1/2} (\pa \D), L^2 (\pa \D) , H^{1/2} (\pa \D),  \wh \gan, (- \ii)\wh \ga_0  ) \\
 \text{ are m-boundary tuples for $\A^*$.} 
\end{multline} 

%We say that these two m-boundary tuples are dual to each other.

The 1st-order acoustic equation \eqref{e:Ac} is not completely equivalent to the 2nd-order version 
\eqref{e:Ac2or}. However, it becomes equivalent after the restriction of \eqref{e:Ac}
to the case where the vector-fields $\ab \vbf$ are gradients of $H^1 (\Om)$-functions.

Let us perform this reduction.
The notation $\ker T := \{y \in \dom T : Ty=0\}$ stands for the kernel of an operator $T:\dom T \subseteq \Xf_1 \to \Xf_2 $.
The kernel $ \ker \Div_0$ of the closed operator $\Div_0$ is a closed subspace of $\LL^2_{\ab} (\D)$ and is denoted by 
\[
\text{ $ \HH_0 (\Div 0,\D) := \{ \ubf \in \HH_0 (\Div,\D) \ : \ \Div \ubf  = 0 \}$.}
\]

We use the following  orthogonal decomposition 
\begin{gather} \label{e:La}
\LL^2_{\ab} (\D) = \HH_0 (\Div 0,\D) \oplus \ab^{-1} \gradm H^1 (\D) , 
\end{gather}
where $ \ab^{-1} \gradm H^1 (\D) := \{ \ab^{-1} \nabla p : p \in H^1 (\D)\}$ is understood as a subspace of $\LL^2_{\ab} (\D)$.
Since $\gradm H^1 (\D)$ is closed in $\LL^2 (\D)$ and $\ab \in L^\infty (\D, \RR_{\sym}^{d\times d})$ is uniformly positive,  the subspace $ \ab^{-1} \gradm H^1 (\D)$ is closed in $\LL^2_{\ab} (\D)$. Thus, $ \ab^{-1} \gradm H^1 (\D)$ is a Hilbert space with the norm of $\LL^2_{\ab} (\D)$.

Consequently, the phase space $ \LL^2_{\ab,\beta} (\D) = \LL^2_{\ab} (\D) \oplus L^2_\beta (\D) $ 
admits  the orthogonal decomposition
\begin{gather} \label{e:LabDec}
\LL^2_{\ab,\beta} (\D) = \HH_0 (\Div 0,\D) \oplus \GG_{\ab,\beta} , % \quad \text{ where ,}
\end{gather}
where 
\[
 \GG_{\ab,\beta} := \ab^{-1} \gradm H^1 (\D) \oplus L^2_{\beta} (\D) .
 \]
Here $ \GG_{\ab,\beta}$ and $\HH_0 (\Div 0,\D)$ are perceived as closed subspaces of  $\LL^2_{\ab,\beta} (\D)$ and so, are supposed to be equipped with the norm of  $\LL^2_{\ab,\beta} (\D)$. 

Note that  $\HH_0 (\Div 0,\D) = \ker \A$ for the closed symmetric operator $\A$ of \eqref{e:A}, and so,
 $\HH_0 (\Div 0,\D)$ is a reducing subspace for  $\A$ and $\A^*$.   
  That is, the decomposition \eqref{e:LabDec} reduces $\A$ to the orthogonal sum
$\A = 0 \oplus \A |_{\GG_{\ab,\beta}}$, where the part 
$\A|_{\HH_0 (\Div 0,\D)}$ of $\A$ in the space $\HH_0 (\Div 0,\D)$ is the zero operator.
Similarly, $\A^* = 0 \oplus \A^* |_{\GG_{\ab,\beta}} = 0 \oplus (\A |_{\GG_{\ab,\beta}})^* $,
where the adjoint $(\A |_{\GG_{\ab,\beta}})^*$ is understood in the sense of the Hilbert space 
$\GG_{\ab,\beta}$.

For the definition of reducing subspaces we refer to the textbook \cite{AG} (more details about reducing subspaces for acoustic operators can be found in \cite[Proposition 6.2]{K24}).

The interplay between the 1st-order and the 2nd-order acoustic equations is shown by the following fact: the parts $ \A |_{\GG_{\ab,\beta}}$ and $ \A^* |_{\GG_{\ab,\beta}}$ of the operators $\A$ and $\A^*$ are unitarily equivalent to the operators $\Bc$ and $\Bc^*$, respectively. Moreover, this induces the connections between the corresponding m-boundary tuples 
$\Tc^1$ and $\Tc^2$ (or between their duals $\Tc_*^1$ and $\Tc_*^2$), and in turn,
induces the unitary equivalence of the m-dissipative restrictions of $\A^*$ and $\Bc^*$
defined by appropriate boundary conditions written in terms of the m-boundary tuples $\Tc^1$ and $\Tc^2$.
These facts are proved in \cite{K24}.

Let us describe briefly this unitary equivalence. 
The  norm in the Hilbert space $H_{1,\ab^{-1}} (\D)$ is defined in such a way that  the map 
\begin{equation} \label{e:aGrad1}
\ab^{-1} \gradm_1 : u \mapsto  \ab^{-1} \nabla u, \qquad \ab^{-1} \gradm_1: H_{1,\ab^{-1}} (\D) \to \ab^{-1} \gradm H^1 (\D),
\end{equation}
is a unitary operator. Note that the Hilbert space $\ab^{-1} \gradm H^1 (\D) $ is equipped with the norm of $\LL^2_{\ab} (\D)$. 
Hence, 
 \begin{gather} \label{e:U}
 \Uc:  \begin{pmatrix} u  \\ p   \end{pmatrix} \mapsto \begin{pmatrix} -\ab^{-1} \nabla u \\ p   \end{pmatrix} 
 \quad \text{ is a unitary operator from $\XX$ onto $\GG_{\ab,\beta} $.}
 \end{gather}
This implies immediately the unitary equivalencies   
$ \Uc^{-1} ( \A|_{\GG_{\ab,\beta}}) \Uc = \Bc $ and 
\[
 \Uc^{-1} (\A|_{\GG_{\ab,\beta}})^* \Uc =   \ii 
\begin{pmatrix} 0 & \Ic_1  \\
\beta^{-1} \Div  \ab^{-1} \gradm_1  & 0  \end{pmatrix},
 \] 
 and so, implies statements (i)-(ii) of Lemma  \ref{l:B}. The remaining statement (iii) of Lemma \ref{l:B} follows now from \eqref{e:mTf}.

\subsection{Generalized impedance boundary conditions and m-dissipativity}
\label{s:OpGenIBC}

Assume that $\Zc:\dom \Zc \subseteq H^{1/2} (\pa \D)\to H^{-1/2} (\pa \D)$ is accretive as an operator from the space $H^{1/2} (\pa \D)$ 
to $ H^{-1/2} (\pa \D)$ in the sense of Section \ref{s:i}.
That is, we assume that $\re \< \Zc h,h\>_{L^2 (\pa \D)} \ge 0$.
Equipping the 2nd-order acoustic differential operation $ \bfr_{\ab,\beta} $ (see \eqref{e:Ac2or}) with the boundary condition 
\begin{gather} \label{e:GIBC2}
\Zc  \ga_0 (p) =  \gan (-\al^{-1} \nabla u)  ,
\end{gather}
we obtain the acoustic operator $\Bc_\Zc: \dom \Bc_\Zc \subset \XX \to \XX$ 
defined as the restriction $\Bc^* \uph_{\dom \Bc_\Zc}$  to 
\[
\dom \Bc_\Zc := \{ \Psi=\{u,p\} \in \dom \Bc^* : \Zc \ga_0 (p) =  \gan (-\al^{-1} \nabla u) \}.
\]
Here  \eqref{e:GIBC2} is called a \emph{generalized impedance boundary condition (GIBC)} associated with the 2nd-order acoustic operator $\Bc_\Zc$. In this context, once can say also that $\Bc_\Zc$ is associated with the  impedance operator $\Zc$, see  \cite{HJN05,K24} (in the case of  Maxwell systems, see  also \cite{ACL18,EK22,EK24}).

\begin{prop}[cf. \cite{EK22,K24}] \label{p:Mdis2}
The following statements are equivalent:
\item[(i)] $\Bc_\Zc$ is m-dissipative in $\XX$;
\item[(ii)] $\Zc$ is densely defined  and maximal accretive ( as an operator from  $H^{1/2} (\pa \D)$ 
to $ H^{-1/2} (\pa \D)$);
\item[(iii)] $\Zc$ is closed and maximal accretive.
\end{prop}

 This proposition follows immediately from the combination of Lemma \ref{l:B} with Proposition \ref{p:m-dis}. 
The equivalence (i) $\Leftrightarrow $ (iii) is implicitly present in \cite{K24}, where its analogue is explicitly proved  
for the 1st order acoustic operators (see also \cite{EK22}, where  an analogue for Maxwell systems was obtained).

\begin{rem} \label{r:SA}
The particular case  $\wh \Bc = \Bc_\Zc^*$ of Proposition \ref{p:Mdis2} takes place if and only if $\Zc^\cross = - \Zc^\cross$,
where $\Zc^\cross: \dom \Zc \subseteq H^{1/2} (\pa \D) \to H^{-1/2} (\pa \D)$ is  the $\cross$-adjoint of $\Zc$ \cite{K24}.
This statement can be easily obtained from  Lemma \ref{l:B} and Proposition \ref{p:m-dis}.
\end{rem}

\begin{prop} \label{p:Mdis3}
Assume that $\Zc \in \Lc (H^{1/2} (\pa \D),H^{-1/2} (\pa \D))$ is accretive. 
Then the acoustic operator $\Bc_\Zc$ is m-dissipative in $\XX$.
\end{prop}

\begin{proof}
Since an accretive operator $\Zc$ is a bounded operator with $\dom \Zc = H^{1/2} (\pa \D)$, it is  
maximally accretive as an operator from  $H^{1/2} (\pa \D)$ 
to $ H^{-1/2} (\pa \D)$. Proposition \ref{p:Mdis2} completes the proof.
\end{proof}

\subsection{Resolvent compactness and the proof of Theorem \ref{t:main}}
\label{s:ProofMT}

First, we obtain the main result of this paper, Theorem \ref{t:main}, from its abstract version, Theorem \ref{t:absCR}. 
To this end, we use the m-boundary tuple 
\[
\Tc^2 :=(H^{1/2} (\pa \D), L^2 (\pa \D) , H^{-1/2} (\pa \D),  \wt \ga_0, \ (-\ii) \wt \gan ) 
% \Tc_*^2 :=(H^{-1/2} (\pa \D), L^2 (\pa \D) , H^{1/2} (\pa \D),  \wt \gan, (- \ii) \wt \ga_0  ) 
\]
of Lemma \ref{l:B}. 

The Neumann-type boundary condition  $\gan (-\ab^{-1} \nabla u) = 0$ is associated with the selfadjoint acoustic operator $\wh \Bc_{\Nr}$,  which is defined by 
\begin{gather*} \label{e:Bnr}
\wh \Bc_{\Nr} = \Bc^* \uph_{\dom \wh \Bc_{\Nr}}, \quad \dom \wh \Bc_{\Nr} := \{ \{u,p\} \in \dom \Bc^* : \gan (-\ab^{-1} \nabla u) = 0\}   
\end{gather*}
and which corresponds to the operator $\wh A_1$ in the abstract settings of Section \ref{s:CompResAbs}.

\begin{lem}[\cite{L13}] \label{l:AN}
Consider the 1st order acoustic operator $\wh \A_{\Nr}$ associated with the Neumann boundary condition 
$\gan (\vbf) = 0$ in the sense that 
\[
\wh \A_{\Nr} := \A^* \uph_{\dom \wh \A_{\Nr} }, \qquad \dom \wh \A_{\Nr} := \{\{\vbf,p\} \in \dom \A^* \ : \ \gan (\vbf) = 0\} .
\]
 Then the orthogonal decomposition 
\begin{gather*} 
\LL^2_{\ab,\beta} (\D) = \HH_0 (\Div 0,\D) \oplus \GG_{\ab,\beta} , % \quad \text{ where ,}
\end{gather*}
reduces $\wh \A_{\Nr}$ to the orthogonal sum
$\wh \A_\Nr = 0 \oplus \wh \A_\Nr |_{\GG_{\ab,\beta}}$, where 
$ \wh \A_\Nr |_{\GG_{\ab,\beta}}$ is a selfadjoint operator in the Hilbert space $\GG_{\ab,\beta}$.
Besides, $ \wh \A_\Nr |_{\GG_{\ab,\beta}}$ has compact resolvent and purely discrete spectrum.
\end{lem}

This lemma follows from \cite[Corollary 7.12]{L13} (see also \cite{K24}).

Due to the unitary equivalence described in Section \ref{s:1stOrder}, Lemma \ref{l:AN} implies that  
\begin{gather} \label{e:BN}
\text{$\wh \Bc_{\Nr}$ has compact resolvent and purely discrete spectrum.}
\end{gather}

\begin{proof}[Proof of Theorem \ref{t:main}.]
After the preparations done above, Theorem \ref{t:main} follows from 
the combination of Theorem \ref{t:absCR}, Lemma \ref{l:B}, Proposition \ref{p:Mdis2}, and \eqref{e:BN}.
\end{proof}

By Remark \ref{r:SA}, the acoustic operator $\Bc_\Zc$ is selfadjoint in $\XX$ if and only if $\Zc^\cross = -\Zc$. In this case,  Theorem \ref{t:main} can be compimented by the following statement.

\begin{cor} \label{c:SiDisc}
Assume that the impedance operator $\Zc$ satisfies $\Zc^\cross = -\Zc$.
Then $\si(\Bc_\Zc) = \si_\disc (\Bc_\Zc)$ if and only if $\Zc$ is compact.
\end{cor}

\begin{proof}
The corollary follows immediately from Remark \ref{r:SA} and Theorem \ref{t:main}.
\end{proof}

\begin{rem} \label{r:ess}
Corollary \ref{c:SiDisc} implies the following statement. If  $\Zc^\cross = -\Zc$ and $\Zc:\dom \Zc \subseteq H^{1/2} (\pa \D)\to H^{-1/2} (\pa \D)$ \quad \emph{is not a compact operator from $H^{1/2} (\pa \D)$ to $H^{-1/2} (\pa \D)$}, then 
\[
\si_\ess (\Bc_\Zc) \neq \varnothing
\] 
for the selfadjoint acoustic operator $\Bc_\Zc$. This statement can be used to construct nontrivial examples 
of acoustic operators with a non-empty essential spectrum.
\end{rem}

\section{Impedance boundary conditions and resolvent compactness}

We turn our attention to the impedance boundary conditions 
\begin{gather} \label{e:IBC2}
\zeta  \ga_0 (p) =  \gan (-\al^{-1} \nabla u)  , 
\end{gather}
with impedance coefficients  $\zeta (x)$, $x \in \pa \D$ (see \cite{L83,HJN05,CK13}).
We always assume that 
$\zeta : \pa \D \to \overline{\CC}_\rr$ takes its values in the closed complex right half-plane  $\overline{\CC}_\rr$ and is measurable  (w.r.t. the surface measure of the Lipschitz boundary $\pa \D$).

\subsection{Bounded impedance coefficients}

Consider first the case, where $\zeta : \pa \D \to \overline{\CC}_\rr$ is additionally essentially bounded, i.e., where $\zeta \in L^{\infty} (\pa \D, \overline{\CC}_\rr)$. Then  \eqref{e:IBC2} can be written as $\Zc  \ga_0 (p) =  \gan (-\al^{-1} \nabla u) $ with the accretive impedance operator $\Zc$ equal to the multiplication operator $\mulz : f \mapsto \zeta f$, which is defined on the whole trace space $H^{1/2} (\pa \D)$
and is considered as an operator mapping into the target space $H^{-1/2} (\pa \D)$. 
In this way the definition of Section \ref{s:OpGenIBC} for the acoustic operator $\Bc_\Zc = \Bc_{\mulz}$ is applicable to the impedance boundary condition \eqref{e:IBC2}.

\begin{cor}
If $\zeta \in L^{\infty} (\pa \D, \overline{\CC}_\rr)$, then the acoustic operator 
$\Bc_{\mulz}$ (associated with \eqref{e:IBC2}) is an m-dissipative operator with a compact resolvent and purely discrete spectrum.
\end{cor}

\begin{proof}
The compact embeddings  \eqref{e:CEs1s2}
imply that  
the embedding operator $\wt \Ic : L^2 (\pa \D) \to H^{-1/2} (\pa \D)$ is compact.
Since $\zeta \in L^{\infty} (\pa \D)$, the multiplication on $\zeta$ can be considered as a bounded operator $\wt \mul_\zeta \in \Lc (H^{1/2} (\pa \D),  L^2 (\pa \D))$. 
Hence, the impedance operator $ \Zc = \mul_\zeta $ can be represented as the product of a bounded and a compact operators, $\Zc = \wt \Ic \wt \mul_\zeta $. Thus, $\Zc$ is compact as an operator from  $H^{1/2} (\pa \D)$ to $ H^{-1/2} (\pa \D)$. 
%The operator $\Zc$ is accretive since $\zeta$ is $\overline{\CC}_\rr$-valued.
The application of Theorem \ref{t:main} completes the proof.
\end{proof}

\subsection{Singular impedance coefficients and the proof of Theorem \ref{t:ImpLq}}
\label{s:SingImp}

Consider now the case where the impedance coefficient $\zeta (\cdot)$ is not essentially bounded.

In this case, one needs a more careful interpretation of the multiplication operator $\mul_\zeta$ from 
$ H^{1/2} (\pa \D)$ to  $H^{-1/2} (\pa \D) $ and of the associated impedance boundary condition \eqref{e:IBC2}.
Note that in the case of Maxwell systems there is a variety of nonequivalent interpretations of impedance boundary conditions, see \cite{K94,LL04} and the discussion in \cite{EK22}.

Let $0< s \le 1$. Consider the class $M^s (\pa \D)$ of measurable functions $f: \pa \D \to \CC$ such that
\begin{gather} \label{e:intA}
\left|\int_{\pa \D}  f g h \right|  \lesssim \| g \|_{H^{s} (\pa \D)} \| h \|_{H^{s} (\pa \D)} \quad \text{ for all } g,h \in H^{s} (\pa \D),
\end{gather}
where the notation `$\lesssim$' means that the corresponding inequality is valid after multiplication of the left (or right)  side  on a certain constant $C>0$ independent of other entries; i.e.,  \eqref{e:intA} means $ | \int_{\pa \D} f g h| \le C  \| g \|_{H^{1/2} (\pa \D)} \| h \|_{H^{1/2} (\pa \D)} $ for all $f,g \in H^{1/2} (\pa \D)$ (this inequality assumes that the Lebesgue  integral 
$\int_{\pa \D}  f g h $
w.r.t. the surface measure of $\pa \D$ exists and is finite for all $f,g \in H^{1/2} (\pa \D)$).  

Assume that $f \in M^{1/2} (\pa \D)$. Then  
\[
\mf_f (g,h) := \int_{\pa \D} \zeta g \overline{h} \ ,   \qquad 
\mf_f: H^{1/2} (\pa \D) \times H^{1/2} (\pa \D)  \to \CC, 
\]
is a bounded sesquilinear form on $H^{1/2} (\pa \D) $. Thus, the sesquilinear form $\mf_f$
defines a unique bounded operator $\mul_f \in \Lc (H^{1/2} (\pa \D), H^{-1/2} (\pa \D) )$ by the formula 
\[
\<\mul_f g,h\>_{L^2 (\pa \D)} = \int_{\pa \D} f g \overline{h} , \qquad g,h \in H^{1/2} (\pa \D).
\]

In other words, the assumption $f \in M^{1/2} (\pa \D)$ is equivalent to the statement that $\mul_f$ is a pointwise multiplier from $H^{1/2} (\pa \D)$ to $ H^{-1/2} (\pa \D)$.
More generally, the assumption $f \in M^s (\pa \D)$ means that 
$\mul_f$ can be considered as a pointwise multiplier from $H^s (\pa \D)$ to $ H^{-s} (\pa \D)$.

\begin{cor} \label{c:M12}
Suppose that $\zeta \in M^{1/2} (\pa \D)$ is $ \overline{\CC}_\rr$-valued. Then the acoustic operator $B_{\mulz}$ is m-dissipative.
\end{cor}

\begin{proof}
The operator $\mulz$ is bounded from $H^{1/2} (\pa \D)$ to $H^{-1/2} (\pa \D)$ since $\zeta \in M^{1/2} (\pa \D)$. Besides, $\mulz$  is accretive since $\zeta$ is $\overline{\CC}_\rr$-valued.
Now, the statement of corollary follows from Proposition \ref{p:Mdis3}. 
\end{proof}

Let $0<s\le 1$ and $ 1\le p <\infty$. Then the compact embedding  
\begin{equation} \label{e:SeLp}
H^s (\pa \D) \imb \imb L^p (\pa \D) \quad \text{ holds if }  \quad 
 \frac{s}{d-1} - \frac{1}{2} >  - \frac{1}{p}.
%\qquad \frac{s}{d-1} \ge \frac{1}{2} - \frac{1}{p}.
\end{equation}
The embedding \eqref{e:SeLp} may be well-known to specialists, cf. the  remarks in \cite[Section 6.4]{T02} and the continuous embedding theorems for Besov spaces in \cite{JW84}. However, in the case of Lipschitz boundaries $\pa \D$ or compact Lipschitz manifolds, the author was not able to find a reference that explicitly contain an embedding of this type  for fractional Sobolev spaces. Therefore, we derive in Appendix \ref{a} a proof of \eqref{e:SeLp} from the Jonsson \& Wallin continuous embedding \cite[Proposition 8.5]{JW84}  for Besov spaces on $(d-1)$-subsets of $\RR^d$.

\begin{lem} \label{l:LqMs}
Let $0<s \le  1$, $q > \max \{\frac{d-1}{2s},1\}$,  and $\zeta \in L^q (\pa \D)$. Then $\zeta \in  M^s (\pa \D)$ and the corresponding multiplication operator $\mulz$ is compact as an operator from $H^s (\pa \D)$ to $H^{-s} (\pa \D)$.
\end{lem}

\begin{proof}
Let us take an arbitrary $\si$ such that $0 < \si  < s $ and $ \si < (d-1)/2$. Let $p, p' \in (1,+\infty)$ be such that $\frac{1}{p} + \frac{1}{p'} =1$ and 
$\frac{1}{p}  = \frac{1}{2} - \frac{\si}{d-1} $.
Then $\frac{s}{d-1} - \frac{1}{2} >  - \frac{1}{p}$ and $p' < 2 < p$. Hence, the continuous embeddings \eqref{e:SeLp} holds. 
%Note that $\frac{\si}{d-1}+ \frac{1}{2}  = \frac{1}{p'}   $.

Let $f \in L^q (\pa \D)$ with $q = \frac{d-1}{2\si}$. Using the Hölder inequality, one gets 
\[
\| f g \|_{L^{p^\prime} (\pa \D)} \le 
\left(\int_{\pa \D} |f|^{p^\prime (p-1)/ (p-2)} \right)^{\frac{p-2}{p^\prime (p-1)} }
\left(\int_{\pa \D} |g|^{p^\prime (p-1)} \right)^{\frac1{p^\prime (p-1)}}
= \|f\|_{L^q (\pa \D)} \| g \|_{L^p(\pa \D)} .
\]
 Hence,
\[
\int_{\pa \D} |f g h| \le \| f g \|_{L^{p^\prime} (\pa \D)} \| h \|_{L^p (\pa \D)} 
\le \|f\|_{L^q (\pa \D)} \| g \|_{L^p (\pa \D)} \| h \|_{L^p (\pa \D)} 
%\le \|f\|_{L^q(\pa \D)} \| g \|_{H^s (\pa \D)} \| h \|_{H^s (\pa \D)} .
\]
(i.e.,
$\mul_f$ is a pointwise multiplier from $L^p (\pa \D)$ to $L^{p^\prime} (\pa \D)$).
 Combining this with  the embedding \eqref{e:SeLp}, 
one obtains 
\[
\int_{\pa \D} |f g h| \lesssim  \|f\|_{L^q(\pa \D)} \| g \|_{H^s (\pa \D)} \| h \|_{H^s (\pa \D)} .
\]
Thus, $f \in M^s (\pa \D)$.  Since the embedding \eqref{e:SeLp} is compact, we see that the corresponding multiplication operator 
$\mulz :H^s (\pa \D) \to H^{-s} (\pa D)$ is compact.

It remains to notice that for arbitrary $q > \max \{\frac{d-1}{2s},1\}$, we can choose $\si$ 
in the interval $0<\si < \min\{s,(d-1)/2)\}$ such that 
$q = \frac{d-1}{2\si}$.
\end{proof}

The following theorem is the main result of this section. Note that it includes Theorem \ref{t:ImpLq} as its part.

\begin{thm} \label{t:SingImp}
Suppose that $\zeta: \pa \D \to \overline{\CC}_\rr$ satisfies at least one of the following conditions:
\begin{itemize}
\item[(i)] $\zeta \in M^s (\pa \D)$ with a certain $s \in (0,1/2)$;
\item[(ii)] $\zeta \in L^q (\pa \D )$ for a certain $q>d-1$. 
\end{itemize}
Then 
$\Bc_{\mulz}$ is an m-dissipative operator with a compact resolvent, and $\si (\Bc_{\mulz}) =
\si_\disc (\Bc_{\mulz})$.
\end{thm}

\begin{proof}
Suppose (i). Let $0<s<1/2$. Then $M^s (\pa \D)\subseteq M^{1/2} (\pa \D)$. Corollary \ref{c:M12} implies that
$\Bc_{\mulz} $ is m-dissipative.  It remains to prove that $\mulz$ is compact as an operator from  $H^{1/2} (\pa \D)$ to $ H^{-1/2} (\pa \D)$ and apply Theorem \ref{t:main}.

Due to the compact embeddings \eqref{e:CEs1s2},
the embedding operators $\Ic_+ : H^{1/2} (\pa \D) \to H^{s} (\pa \D)$ and 
$ \Ic_- : H^{-s} (\pa \D) \to H^{-1/2} (\pa \D)$ are compact.
Since $\zeta \in M^s (\pa \D)$, the operator $\wt \mul_\zeta : g \mapsto \zeta g$ 
is a pointwise multiplier from $H^s (\pa \D)$ to $ H^{-s} (\pa \D)$, i.e.,
$\wt \mul_\zeta  \in \Lc (H^s (\pa \D),  H^{-s} (\pa \D))$. 
The impedance operator $ \Zc = \mul_\zeta $ can be represented as 
$\Zc = \Ic_- \wt \mul_\zeta \Ic_+$ and so is compact. Theorem \ref{t:main} completes the proof.

Suppose (ii). Then Lemma \ref{l:LqMs} with $s=1/2$ and Theorem \ref{t:main} imply the desired statement (alternatively, one can check that (ii) implies (i)).
\end{proof}

\section{Conclusion and discussion}
\label{s:dis}

In Theorem \ref{t:main}, we provided a complete description of impedance operators $\Zc$ that generate  acoustic operators   $\Bc_\Zc$ with compact resolvent.

In the case of impedance boundary conditions \eqref{e:IBC2}, Theorem \ref{t:main} naturally leads to the question of \emph{the characterization of all compact pointwise multipliers  from $H^{1/2} (\pa \D)$ to $H^{-1/2} (\pa \D)$}. We do not know how close is the sufficient condition of Theorem \ref{t:ImpLq} to the optimal result
in terms of $L^q (\pa \Om)$-spaces. 

It should be noticed that the case $d=2$ is especially interesting since the 2-dimensional acoustic equation appears as a dimensionally reduced Maxwell equation in 2-D photonic crystals \cite{FK96,ACL18}. Descriptions of wide classes of m-dissipative boundary conditions associated with discrete spectra are useful for rigorous randomization of absorbing boundary conditions, and so, for  the modeling of leakage of energy to a uncertain or stochastic surrounding medium, see the discussions in \cite{DMTSH14,EK22,K24}.

%\section*{Appendix}
\appendix

\section{Appendix: embeddings for spaces of fractional order on $\pa \D$}
\label{a}

Let $0<s \le 1$, $1\le p <\infty$, $\frac{1}{2} - \frac{s}{d-1} < \frac{1}{p}$,
and $\D \subset \RR^d$ be a Lipschitz domain.
The goal of this section is to give a proof of \eqref{e:SeLp}, i.e.,
of  the compact embedding \eqref{e:SeLp}
$H^s (\pa \D) \imb \imb L^p (\pa \D)  $.

Due to the compact embeddings \eqref{e:CEs1s2} and the standard continuous embeddings of $L^p $-spaces on the compact Lipschitz boundary $\pa \D$, it is enough to prove 
the continuous embedding 
\begin{equation} \label{e:SeLpA}
H^{s_1} (\pa \D) \imb L^p (\pa \D) 
\end{equation}
for $2 < p<\infty$ and $0<s_1 <s $ such that $\frac{1}{2} - \frac{s_1}{d-1} < \frac{1}{p}$.

For $\si >0$ and $1 \le q \le +\infty$, consider following \cite{JW84} the Besov spaces $B^\si_{q,q} (\pa \D)$  on $\pa \D$ (cf. \cite[Remark 6.4]{T02} and \cite[Chapter 7]{T92}). Note that 
the Lipschitz boundary $\pa \D$ is a closed $(d-1)$-set in the sense of \cite[Sections 2.1.1]{JW84}.
We derive \eqref{e:SeLpA} from Proposition 8.5 of \cite{JW84}. This proposition implies that,  for 
$0 < \si_2 \le \si_1 <+\infty$ and $1 \le q_1 \le q_2 \le +\infty$,  
\begin{equation} \label{e:Bimb}
 B^{\sigma_1}_{q_1,q_1} (\pa \D)  \imb B^{\sigma_2}_{q_2,q_2} (\pa \D) \text{ if $ \si_1  - \frac{d-1}{q_1} \ge  \si_2  -  \frac{d-1}{q_2}$}.
\end{equation}
Note that the  Jonsson \& Wallin results of \cite[Chapter 5]{JW84} concerning the traces 
of $B^\si_{q,q} (\RR^d)$ on $\wt d$-sets with $\wt d \in (0,d]$ essentially reduce \eqref{e:Bimb} to the well-known embedding theorem for Besov spaces in $\RR^d$ (see \cite{T83,JW84}).

Let us prove \eqref{e:SeLpA}. Since $0<s_1 < s \le 1$, the comparison of the standard definition of the space $H^{s_1} (\pa \D)$ in \cite[Section 3.2.1]{M03} with the results of \cite[Sections 5.1.1]{JW84} implies that $H^{s_1} (\pa \D) = B^{s_1}_{2,2} (\pa \D)$ up to equivalence of the norms. 
Let $s_2 \in (0,s_1) $ be such that $\frac{1}{2} - \frac{s_2}{d-1} = \frac{1}{p}$.
Applying \eqref{e:Bimb} with  $\si_1 = s_1 $, $q_1 = 2$, $\si_2 =s_1- s_2$, and $q_2 = p$, 
we get 
\begin{equation} \label{e:Bimb2}
H^{s_1} (\pa \D) = B^{s_1}_{2,2} (\pa \D) \imb    B^{s_1 -s_2 }_{p,p} (\pa \D)  .
\end{equation}
Since $0<s_1 -s_2<1$, the formula for the norm of $B^{s_1 -s_2 }_{p,p} (\pa \D)$ in \cite[Section 5.1.1]{JW84} yields $B^{s_1 -s_2 }_{p,p} (\pa \D)  \imb L^p (\pa \D)$.
Thus, \eqref{e:SeLpA} holds true. This completes the proof of \eqref{e:SeLp}.

%\vspace{3ex}
%\noindent
%\emph{Acknowledgments.} 

\end{document}